\documentclass[11pt,a4paper]{amsart}

\usepackage[margin=3cm,top=2.5cm,bottom=2cm]{geometry}

\usepackage{amsmath, amssymb, amsthm, mathrsfs}
\usepackage{graphicx,enumitem,microtype} 
\usepackage[dvipsnames]{xcolor}

\usepackage[colorlinks=true,linkcolor=red!70!black,urlcolor=MidnightBlue,citecolor=MidnightBlue]{hyperref}

\usepackage[utf8]{inputenc} 
\usepackage[T1]{fontenc}
\usepackage{lmodern}

\newcommand\R{{\mathbb R}}

\newcommand{\C}{\mathbb C}

\def\BBB{{\mathscr{B}}}

\def\MMM{{\mathscr{M}}}

\renewcommand\AA{\mathcal A}
\newcommand\BB{\mathcal B}

\newcommand\DD{\mathcal D}

\newcommand\FF{\mathcal F}

\newcommand\KK{\mathcal K}
\newcommand\LL{\mathcal L}

\newcommand\OO{\mathcal O}

\newcommand\RR{\mathcal R}

\newcommand\UU{\mathcal U}

\newcommand\XX{\mathcal X}
\newcommand\YY{\mathcal Y}
\newcommand\ZZ{\mathcal Z}

\newcommand\eps{\varepsilon}

\newcommand{\beqn}{\begin{equation}}
\newcommand{\eeqn}{\end{equation}}
\newcommand{\bear}{\begin{eqnarray}}
\newcommand{\eear}{\end{eqnarray}}
\newcommand{\bean}{\begin{eqnarray*}}
\newcommand{\eean}{\end{eqnarray*}}
\newcommand{\bal}{\begin{aligned}}
\newcommand{\eal}{\end{aligned}}

\newcommand{\la}{\langle}
\newcommand{\ra}{\rangle}
\newcommand{\rra}{\rangle\!\rangle}
\newcommand{\lla}{\langle\!\langle}

\renewcommand{\d}{\mathrm{d}}

\DeclareMathOperator{\Div}{div}

\numberwithin{equation}{section}

\setlist[enumerate]{wide,labelindent=0cm,label=\textnormal{(\arabic*)},itemsep=5pt,topsep=4pt}

\theoremstyle{plain}
\newtheorem{theo}{Theorem}[section]
\newtheorem{prop}[theo]{Proposition}
\newtheorem{lem}[theo]{Lemma}
\newtheorem{cor}[theo]{Corollary}
\theoremstyle{remark}
\newtheorem{rem}[theo]{Remark}
\theoremstyle{definition}

\title[Stability of the parabolic-parabolic Keller-Segel equation]{On the self-similar stability of the parabolic-parabolic Keller-Segel equation}

\author[F. Alvarez]{Frank Alvarez Borges}
\address[Frank Alvarez Borges]{Laboratoire Jacques-Louis Lions (LJLL),
  Sorbonne Universit\'e UPMC, 4 Place Jussieu, 75005 Paris 5, France}
\email{frank.alvarez$\_$borges@sorbonne-universite.fr}

\author[K. Carrapatoso]{Kleber Carrapatoso}
\address[Kleber Carrapatoso]{Centre de Math\'ematiques Laurent Schwartz, \'Ecole Polytechnique, Institut Polytechnique de Paris, 91128 Palaiseau cedex, France}
\email{kleber.carrapatoso@polytechnique.edu}

\author[S. Mischler]{Stéphane Mischler}
\address[Stéphane Mischler]{Centre de Recherche en Math\'ematiques de
  la D\'ecision (CEREMADE),
  Universit\'es  Paris Dauphine-PSL, Place de Lattre de
  Tassigny, 75775 Paris 16, France \& Institut Universitaire de France (IUF)}
\email{mischler@ceremade.dauphine.fr}

\date{\today}

\keywords{Parabolic-parabolic Keller-Segel equation, self-similar variables, large-time behaviour, nonlinear stability}

\subjclass[2020]{35B35, 35B25, 35K45, 35K55, 35Q92}

\begin{document}

\begin{abstract}
We consider the parabolic-parabolic Keller-Segel equation in the plane and prove the nonlinear exponential stability of the self-similar profile in a quasi parabolic-elliptic regime. We first perform a perturbation argument in order to obtain exponential stability for the semigroup associated to part of the first component of the linearized operator, by exploiting the exponential stability of the linearized operator for the parabolic-elliptic Keller-Segel equation.
We finally employ a purely semigroup analysis to prove linear, and then nonlinear, exponential stability of the system in appropriated functional spaces. 

\end{abstract}

\maketitle

\tableofcontents

\section{Introduction}

In this paper we are concerned with the parabolic-parabolic Keller-Segel system in self-similar variables  in the plane 
\begin{equation}\label{eq:KS_ssChap7}
\left\{
\begin{aligned}
\partial_t f  & = \Delta f + \Div ( \mu x f - f \nabla u) \\
\partial_t u &= \frac{1}{\eps} (\Delta u + f) +   \mu x \cdot \nabla u,
\end{aligned}
\right.
\end{equation}
with fixed drift parameter $\mu > 0$ and with  small time scale parameter $\eps >0$,  
which aims to give the time evolution of the collective motion of cells (described by the {\it cells density} $f=f(t,x)$) that are attracted by a chemical substance (described by 
the {\it chemo-attractant concentration}  $u = u(t,x)$)
they are able to emit (\cite{P,KS}).  Here $t \ge 0$ is  the time variable and $x \in \R^2$ stands for the space variable. 
 We refer to the work \cite{calvez2008parabolic} as well as to the reviews \cite{MR2448428,MR2430317} and the references quoted therein for biological motivation and mathematical introduction.

\smallskip
In this work, we establish  in a convenient weighted Sobolev space the exponential stability of the {\it normalized self-similar profile} in the quasi parabolic-elliptic regime, that is for small values of the time scale $\eps > 0$, without assuming any radial symmetry property on the initial datum. This extends similar results obtained 
in \cite{MR3615547} 
in a radially symmetric framework.
As in that last reference, the proof of the stability is based on a perturbation argument which takes advantage of the  exponential stability of the self-similar profile for the parabolic-elliptic Keller-Segel equation established in \cite{MR3196188,MR3466844}. 
The proof however differs from \cite{MR3615547} because  it uses among other things (1) a different, and somehow more standard, perturbation argument performed at the level of the main part of the first component of the linearized operator instead of at the level of the whole linearized system  
and (2) a purely semigroup analysis of the linear and nonlinear stability of the system. 
 
Our result implies that in a quasi-parabolic-elliptic regime and  for some class of initial data without assuming any radial symmetry property, 
the associated solution  to the parabolic-parabolic KS system in standard variables (corresponding to $\mu=0$) has a self-similar long-time behavior, which in particular means that no concentration occurs in large time  and thus the diffusion mechanism is really the dominant phenomenon all along the time evolution. 

It is worth mentioning that as far as the existence problem is concerned, an alternative possible approach has been developed in \cite{calvez2008parabolic} where weak solutions have been proved to exist for a very general class of  initial data, see also \cite{MR2847229,MR3116019}. The associated uniqueness result has been solved in \cite{MR3615547}, see also \cite{MR1657160,MR3117843}, but the accurate analysis of the long-time behavior of these solutions is still lacking. On the other hand, mild solutions have been proved to exist under a smallness condition in the initial datum for instance in \cite{MR2515582,MR2785976,MR3227285} with associated self-similar behavior result  in the longtime asymptotic in  \cite{MR1929883,MR2295185,MR3227285}  or under a large time scale parameter  for instance in \cite{MR3227285,MR3411100}.

The two, and only two, general properties satisfied (at least formally) by the solutions of the parabolic-parabolic Keller-Segel equation are the positivity preservation of the cells density, i.e. 
$$
f(t,\cdot) \ge 0 \quad\hbox{if} \quad 
f(0,\cdot) \ge 0,
$$
and a similar positivity property for the chemo-attractant concentration $u$, as well as the mass conservation of the cells density, namely
\beqn\label{eq:massconservation}
\lla f(t,\cdot) \rra  = 
\lla f(0,\cdot) \rra , \quad \forall \, t \ge 0, \quad 
\lla h\rra  := \int_{\R^2} h dx. 
\eeqn
That mass conservation \eqref{eq:massconservation} is known to be violated in some {\it supercritical mass} situation. However, we will only be concerned in this paper with a  {\it subcritical mass framework} that we describe now. 
  
\smallskip 
We denote by $Q=Q_{\eps}^\mu$ and $P=P_\eps^\mu$   
the normalized stationary solutions to the Keller-Segel system \eqref{eq:KS_ssChap7}, that is 
\begin{equation}\label{eq:KS_ssQPChap7}
\left\{
\begin{aligned}
0  & = \Delta Q + \Div (\mu x Q -Q \nabla P), \qquad Q(0) = 8,  \\
0 &=  \Delta P + Q  + \eps\mu x \cdot \nabla P,
\end{aligned}
\right.
\end{equation}
which existence, uniqueness, radially symmetric property and smoothness have been established in \cite{MR1929883,MR2806487,MR3227285}. 
It is worth emphasizing that we adopt here the normalizing convention of \cite{MR3861072} motivated by the fact that in the vanishing drift limit 
$$
Q_{\eps}^\mu \to Q^0, \quad \nabla P_{\eps}^\mu \to \nabla P^0, \quad\hbox{as}\quad \mu \to 0, 
$$
where $(Q^0,P^0)$ is defined by  
$$
Q^0 (x) := \frac{8}{(1+|x|^2)^2}, \quad \Delta P^0 = Q^0 , 
$$
and thus 
$Q^0$ is the well-known $8\pi$ critical mass solution to the parabolic-elliptic Keller-Segel equation in standard variables (corresponding thus to $\mu=\eps=0$). 
Because for any $\eps  \ge  0$, there exists a one-to-one mapping 
$$
\MMM_\eps: {  [}0,\infty) \to (0,8\pi{  ]}, \quad \mu \mapsto \MMM_\eps(\mu) : = \lla Q^\mu_\eps \rra, 
$$
another possible (and more standard) normalization convention should be to fix the drift term $\mu := 1$ and to normalize the stationary solution by its subcritical mass in
the interval $(0,8\pi)$.

\medskip
We next introduce the   perturbation $(g,v)$ of the stationary state $(Q,P)$ defined by
$$
f= Q + g, \quad \lla g \rra = 0, \quad u = P + v,
$$
in such a way that the mass compatibility condition $\lla f \rra = \lla Q \rra$ is satisfied.
If $(f,u)$ is a solution to  \eqref{eq:KS_ssChap7} then $(g,v)$  satisfies the system
\begin{equation}\label{eq:KS_perturbationChap7}
\left\{
\begin{aligned}
\partial_t g  & = \Delta g + \Div (\mu x g - g \nabla P - Q \nabla v) - \Div (g \nabla v) \\
\partial_t v &= \frac{1}{\eps} (\Delta v + g) + \mu x \cdot \nabla v, 
\end{aligned}
\right.
\end{equation}
and reciprocally. Instead of working with solutions $(g,v)$ to \eqref{eq:KS_perturbationChap7} we shall rather work with the unknown $(g,w)$ defined by 
$$
w := v - \kappa * g,
$$
where $\kappa$ is the Laplace kernel in the plane 
\beqn\label{eq:def_kappa}
\kappa (z) := - \frac{1}{2\pi} \log |z|, \quad   \nabla\kappa(z) 
= - \frac{1}{2\pi} \frac{z}{|z|^2}, 
\eeqn
so that $ \kappa * \Omega$ is a solution to the Laplace equation $
- \Delta   (\kappa * \Omega) = \Omega$ in $\R^2$.  
We will therefore consider the  modified system
\begin{equation}\label{eq:gw}
\left\{
\begin{aligned}
\partial_t g  & = \Delta g + \Div (\mu x g - g \nabla P - Q \nabla \kappa*g - Q \nabla w) - \Div (g \nabla \kappa*g) - \Div (g \nabla w) \\
\partial_t w &= \frac{1}{\eps} \Delta w  + \mu x \cdot \nabla w + g 
\\
&+ \nabla \kappa * \left[ g \nabla P + Q \nabla \kappa*g + Q \nabla w \right] +  \nabla \kappa * \left[ g \nabla w + g \nabla \kappa *g \right] ,
\end{aligned}
\right.
\end{equation}
satisfied by   $(g,w)$,  that we complement  with an initial condition $(g_0 , w_0)$. It is worth mentioning that working with the unknown $w$ is convenient in order to get some nice information (improving estimates) in the limit $\eps \to 0$ what it is not the case when considering the 
unknown $v$ because of the singular term  $\frac{g}{\eps}$ at the right-hand side of the second equation of~\eqref{eq:KS_perturbationChap7} in the limt $\eps \to 0$.

\smallskip
We introduce the Banach spaces $\XX:=L^2_k \times (\dot H^s \cap \dot H^1)$ and {$\YY=H^1_k \times (\dot H^s \cap \dot H^2) $}, with $s \in (0,1)$, endowed with the norms
\begin{equation}\label{eq:def:XX&YY} 
\begin{aligned}
\| (g,w) \|_\XX &:= \| g \|_{L^2_k} + \| w \|_{\dot H^s}+\| w \|_{\dot {H}^1},\\
\| (g,w) \|_\YY &:= \| g \|_{H^1_k} + \| w \|_{\dot H^s} +\| w \|_{\dot {H}^2},
\end{aligned}
\end{equation}
where the weighted Lebesgue space $L^p_k(\R^2)$, for $1 \le p \le \infty$ and  $k \ge 0$, is defined by 
$$
L^p_k(\R^2) := \{ f \in L^1_{\rm loc}(\R^2); \,\, \| f \|_{L^p_k} := \| \langle x \rangle^k f  \|_{L^p} < \infty \}, \quad \langle x \rangle := (1+|x|^2)^{1/2},
$$
the norm of the higher-order Sobolev spaces $H^1_k (\R^2)$ is defined by 
$$
\| f \|_{H^1_k}^2 := \sum_{|\alpha| \le 1} \| \la x \ra^k \, \partial^\alpha f \|_{L^2}^2.
$$
and, for any $\sigma \in   \R_+$, the  homogeneous seminorm $\| \cdot \|_{\dot H^\sigma}$ and  the associated homogeneous Sobolev space $\dot H^\sigma$ are defined by 
$$
  \dot H^\sigma = \{ f \in \mathcal S'(\R^2) \mid \widehat f \in  L^1_{\rm loc}(\R^2), \  \| f \|_{\dot H^\sigma} :=  \|   | \xi |^\sigma \widehat f \|_{L^2} < \infty \}.
$$
We recall that,   on the plane, $\dot H^{\sigma}$ is a Hilbert space for $\sigma < 1$ (see \cite[Proposition 1.34]{MR2768550}).
We also denote by $H^{-1}_k$ the duality space of $H^1_k$  for the scalar product $\langle \cdot, \cdot \rangle_{L^2_k}$, namely
$$
\| \phi \|_{H^{-1}_k} = \sup_{\| f \|_{H^1_k} \le 1} \la \phi , f \ra_{L^2_k}  = \sup_{\| g \|_{H^1}\le 1} \left\la \langle x \rangle^k \phi , g \right\ra_{L^2} = \| \langle x \rangle^k \phi \|_{H^{-1}}, 
$$ 
so that we may identify 
$ H^{-1}_k = \bigl\{ F_0 + \Div F_1; \  F_i \in L^2_k
\bigr\}. 
$
For $k > 1$, so that $L^2_k \subset L^1$, we finally denote
$$
L^2_{k,0} : = \left\{ f \in L^2_k; \; \lla f \rra =  0 \right\}. 
$$

We may now  state our main result. 

\begin{theo}\label{theo:nonlinear_stabilityIntro}
Let us fix $\mu \in (0,\infty)$, $k > 3$ and $s \in (0,1)$. 
There are $\eps_0 , \eta_0>0$ such that for any $\eps \in (0,\eps_0)$ and any initial data $(g_0, w_0 )\in   \XX$ with $ \lla g_0 \rra =  0$ and  
  $\| (g_0, w_0) \|_{\XX} \le \eta_0$, there exists a unique global solution $(g,w)  \in L^\infty_t (\XX) \cap L^2_t (\YY) $ to~\eqref{eq:gw} which verifies
\begin{equation}\label{eq:theo:energyIntro}
\| (g,w) \|_{L^\infty_t (\XX)} + \| (g,w) \|_{L^2_t (\YY)} \lesssim \| (g_0, w_0) \|_{\XX}.
\end{equation}
Moreover, for any $\lambda \in (0,\mu(1-s))$,  we have the  exponential decay estimate 
\begin{equation}\label{eq:theo:decayIntro}
\|  (g(t),w(t)) \|_{\XX} \lesssim e^{-\lambda t} \| (g_0, w_0) \|_{\XX}, \quad \forall \, t \ge 0.
\end{equation}
\end{theo}

This result improves \cite[Theorem~1.4]{MR3615547} where similar estimates are established with the restriction that the initial datum is radially symmetric (and satisfies additional regularity and confinement conditions) and also improves  \cite{MR1929883,MR2295185,MR3227285} which deal with small initial data and arbitrary time scale parameter $\eps >0$.
 It is worth emphasizing that because of the one-to-one mapping $\mu \mapsto \MMM_\eps(\mu)$, the choice of a given drift parameter $\mu \in (0,\infty)$ and its associated steady state $Q^\mu_\eps$ here is equivalent to the choice of a given subcritical mass in $(0,8\pi)$ for the initial datum in \cite{MR3615547,MR1929883,MR2295185,MR3227285}.
 It is also worth underlining that Theorem~\ref{theo:nonlinear_stabilityIntro} implies that the corresponding solution $(F,U)$ to the parabolic-parabolic Keller-Segel equation in standard variables (i.e.\ $\mu=0$) satisfies
$$
F(t,x) \sim \frac1{R(t)^2} \, Q \left(   \frac{x}{R(t)} \right), \quad
U(t,x) \sim    P \left( \frac{x}{R(t)} \right), \quad\text{as}\quad t \to \infty,
$$
with $R(t) := (1+\mu t)^{1/2}$,  and we refer to \cite[Sec~1.]{MR3615547} for further discussions. 

\medskip

As said above, 
we shall always work with the unknown $(g,w)$ and it is worth stressing why the corresponding evolution system is given by \eqref{eq:gw}. We may indeed   observe  that if $(g,v)$ satisfies \eqref{eq:KS_perturbationChap7} then the function $w:=v-\kappa * g$
straightforwardly satisfies 
\bean
\partial_t w 
&=&  \frac1\eps \Delta w + \mu x \cdot \nabla w + \mu x \cdot \nabla \kappa*g - \nabla \kappa * [\nabla g +  \mu x g - g \nabla P - Q \nabla \kappa*g - Q \nabla w] \\
&+&  \nabla \kappa * \left[ g \nabla w + g \nabla \kappa *g \right]. 
\eean
Using that 
\bean
x \cdot \nabla \kappa*g - \nabla \kappa * (x g )  
= - \frac{1}{2 \pi} \int \frac{(x-y)}{|x-y|^2} \bigl\{ x \, g(y) - y \, g(y) \bigr\} \, dy  
= - \frac{1}{2 \pi}  \int_{\R^2} g(y) \, dy = 0,
\eean
because of the mass vanishing condition on $g$, 
the equation on $w$ simplifies and thus $(g,w)$ satisfies~\eqref{eq:gw}. Equivalently, defining the operator 
$$
\LL (g,w) = ( \LL_1(g,w) , \LL_2(g,w) )
$$
by
$$
\LL_i (g,w) = \LL_{i,1} g + \LL_{i,2} w, \quad i = 1,2, 
$$
with 
\begin{equation}\label{eq:KS_defLij}
\left\{
\begin{aligned}
\LL_{1,1} g &= \Delta g + \Div (\mu x g - g \nabla P - Q \nabla \kappa*g ), \quad 
\LL_{1,2} w = - \Div (Q \nabla w), \\
\LL_{2,1} g &=  g + \nabla \kappa * \left[ g \nabla P + Q \nabla \kappa*g \right], \quad
\LL_{2,2} w = \frac{1}{\eps} \Delta w  + \mu x \cdot \nabla w + \nabla \kappa * \left[  Q \nabla w \right], 
\end{aligned}
\right.
\end{equation}
the system~\eqref{eq:gw} on $(g,w)$ rewrites as 
\begin{equation}\label{eq:KS_perturbation_gw_bis}
\left\{
\begin{aligned}
&\partial_t (g,w) = \LL(g,w)  + \left( - \Div (g \nabla \kappa*g) - \Div (g \nabla w),  \nabla \kappa * \left[ g \nabla w + g \nabla \kappa *g \right]  \right)\\
& (g,w)_{|t=0} = (g_0 , w_0).
\end{aligned}
\right.
\end{equation}

\medskip

In Sections \ref{sec:estimQP} and \ref{sec:Fineq}, we present some estimates on the family of steady states $(Q,P)$ and some functional inequalities that will be useful throughout the paper. 
In Section~\ref{sec:EstimL11}, we establish the dissipativity of the operator $\LL_{1,1}$ and next the exponential decay of the associated semigroup for small enough values of $\eps > 0$,
thanks to a perturbation argument and by taking advantage of the dissipativity of the limit operator for $\eps = 0$ corresponding to the usual linearized parabolic-elliptic Keller-Segel operator. 
In Section~\ref{sec:EstimL12}, we prove in a more direct way the dissipativity of the operator $\LL_{2,2}$.
In Section~\ref{sec:EstimSGL}, we deduce then the decay of the semigroup $S_\LL$ associated to $\LL$ by writing in a proper accurate enough semigroup way the two decay estimates of $S_{\LL_{i,i}}$
and by showing that both out of the diagonal contributions $\LL_{i,j}$, $i\not=j$, are small enough. 
The above two arguments significantly differ from those used in the proof of  \cite[Theorem~1.4]{MR3615547}. 
In Section~\ref{sec:proofNLstab}, we finally present the proof of Theorem~\ref{theo:nonlinear_stabilityIntro} which is based on a classical nonlinear stability trick.

\medskip
In the sequel,  for two functions $S$ and $T$ defined on $\R_+$, we define the convolution $S*T$ by 
$$
(S*T)(t) = \int_0^t S(t-s) T(s) \, \d s, \quad \hbox{for all } t \ge 0,
$$
so that in particular the Duhamel formula associated to an evolution equation 
$$
\partial_t g = \Lambda g + G, \quad g(0) = g_0,
$$
writes 
$$
g= S_\Lambda g_0 + S_\Lambda * G.
$$
Moreover, for $\lambda \in \R$, we denote $ \mathbf{e}_\lambda : t \mapsto e^{\lambda t}$. We also write $A \simeq B$ if $A = c B$ for a numerical constant $c$ and $A \lesssim B$ when $A \le c B$ for a numerical constant $c>0$ and $A,B \ge 0$.

\section{Estimates over $Q$ and $P$}
\label{sec:estimQP}
 
We present some estimates on the steady states $Q = Q^\mu_\eps$ and $P=P^\mu_\eps$, that we recall satisfy~\eqref{eq:KS_ssQPChap7}, which will be useful in the next sections.

\begin{prop}\label{UniConChap7}
There exist $\varepsilon_0>0$ and   $\alpha_0 \in (0,1)$ such that  for all $\varepsilon\in(0,\varepsilon_0)$ 
 and $\alpha\in(\alpha_0,1)$ we have:
\begin{enumerate}
   \item (Bounds over $P$) For all $x\in\R^2$ there holds
\bear\label{CompP7}
   &&P^0(x)-\frac{\mu\alpha |x|^2}{2}<P(x)-\frac{\mu\alpha |x|^2}{2}<P^0(x)<P(x)<0,
   \\
   &&\label{CompdP7}
       x\cdot\nabla P(x)-\mu\alpha|x|^2< x\cdot\nabla P^0(x)< x\cdot\nabla P(x)<0.
\eear
   \item (Bounds over $Q$) For all $x\in\R^2$, there holds
\beqn 
   Q^0(x)e^{-\mu\frac{|x|^2}{2}}<Q(x)<Q^0(x)e^{-\mu(1-\alpha)\frac{|x|^2}{2}}.\label{eq:estimQunifChap7}
\eeqn
\end{enumerate}
\end{prop}

\begin{proof}[Proof of Proposition \ref{UniConChap7}]
In order to prove \eqref{CompP7} and \eqref{CompdP7}, we follow the same ideas as in the proof of \cite[Proposition 4.1]{MR3861072}, but including the necessary modifications to handle the terms depending on $\varepsilon$ that appear for this new problem.

First notice that $P^0$ and $P$ are radial functions solving the equations
\begin{align*}
    &\Delta P^0(x)=-Q^0(x)=-8e^{P^0(x)}, 
    \quad  P^0(0)= 0 ,
    \\ 
    &\Delta P (x)+\mu\varepsilon x\cdot\nabla P(x)=-Q(x)=-8e^{P(x)-\mu\frac{|x|^2}{2}},
       \quad  P(0)= 0.
\end{align*}
In polar variables, these equations read as
\begin{align}
    & (P^0)''(r)+\frac{1}{r}(P^0)'(r)=-8e^{P^0(r)}, 
    \quad P^0(0)= (P^0)'(0)=0, \nonumber\\
    &(P)''(r)+\left(\frac{1}{r}+\mu\varepsilon r\right) (P)'(r)=-8e^{P(r)-\mu\frac{|r|^2}{2}},
    \quad P(0)= P'(0)=0.\label{Pmepol}
\end{align}
Solving these two equations, we  get
\begin{align}
    P^0(r)&=-8\int_0^r\frac{1}{\rho}\int_0^{\rho}\tau e^{P^0(\tau)}d\tau d\rho,\label{PolImP}\\
    P(r)&=-8\int_0^r\frac{e^{-\mu\varepsilon\frac{\rho^2}{2}}}{\rho}\int_0^{\rho}\tau e^{P(\tau)-\mu(1-\varepsilon)\frac{\tau^2}{2}}d\tau d\rho.\label{PolImPme}
\end{align}
Plugging an expansion in powers of $r$ up to order 4 for $P(r)$ in \eqref{Pmepol}, the coefficients of such expansion can be computed, which gives
\begin{align*}
    P(r)-\mu\alpha\frac{r^2}{2}&=-(2+\frac{\mu\alpha}{2})r^2+(1+\frac{\mu}{4}(1+\varepsilon))r^4+o(r^4),\\
    P^0(r)&=-2r^2+r^4+o(r^4),\\
    P(r)&=-2r^2+(1+\frac{\mu}{4}(1+\varepsilon))r^4+o(r^4),
\end{align*}
 for any given $\alpha\in(0,1)$.  This implies that there exists  $r_0 = r_0(\alpha)>0$ 
such that 
the following relation holds true
\beqn\label{eq:estimP0-r2P-r2P0P0}
P^0(r)-\frac{\mu\alpha |r|^2}{2}<P(r)-\frac{\mu\alpha |r|^2}{2}<P^0(r)<P(r)<0, \quad \forall \, r \in (0,r_0). 
\eeqn
Set $\alpha=1-\varepsilon$ and assume now, by contradiction,  that there exists $r_1>0$ such that $P^0(r_1)=P(r_1)$ and 
$$
P(r)-\frac{\mu(1-\varepsilon) |r|^2}{2}<P^0(r)<P(r),  \quad \forall \, r \in (0,r_1). 
$$
Using \eqref{PolImP} and \eqref{PolImPme}, we get
\begin{align*}
    0&=P^0(r_1)-P(r_1)\\
    &=-8\int_0^{r_1}\frac{1}{\rho}\int_0^{\rho}\tau \left(e^{P^0(\tau)}-e^{P(\tau)-\mu(1-\varepsilon)\frac{\tau^2}{2}-\mu\varepsilon\frac{\rho^2}{2}}\right)d\tau d\rho\\
    &\le -8\int_0^{r_1}\frac{1}{\rho}\int_0^{\rho}\tau \left(e^{P^0(\tau)}-e^{P(\tau)-\mu(1-\varepsilon)\frac{\tau^2}{2}}\right)d\tau d\rho< 0, 
\end{align*}
the last strict inequality being due to the second inequality in \eqref{eq:estimP0-r2P-r2P0P0}. 
This is a contradiction and therefore $P_0(r)<P(r)$ for all $r>0$.

On the other hand, suppose by contradiction again that there exist $\alpha \in (0,1)$ and $r_\alpha >0 $ such that 
$$
\Psi^{\mu}_{\varepsilon}(\alpha,r_\alpha):=P(r_\alpha)-\frac{\mu\alpha |r_\alpha|^2}{2}=P^0(r_\alpha)
$$
and
$$
P(r)-\frac{\mu \alpha|r|^2}{2}<P^0(r)<P(r), \quad \forall \, r \in (0,r_\alpha).
$$
Using   \eqref{PolImPme}, we have
\begin{align*}
    \Psi^{\mu}_{\varepsilon}(\alpha,r_\alpha)&=-8\int_0^{r_\alpha}\frac{1}{\rho}\int_0^{\rho}\tau e^{P(\tau)-\mu(1-\varepsilon)\frac{\tau^2}{2}-\mu\varepsilon\frac{\rho^2}{2}}d\tau d\rho-\frac{\mu\alpha |r_\alpha|^2}{2}\\
    &<-8\int_0^{r_\alpha}\frac{1}{\rho}\int_0^{\rho}\tau e^{P^0(\tau)-\mu(1-\varepsilon)\frac{\tau^2}{2}-\mu\varepsilon\frac{\rho^2}{2}}d\tau d\rho-\frac{\mu\alpha |r_\alpha|^2}{2}=:\tilde{\Psi}(\alpha,r_\alpha,\mu).
\end{align*}
Notice that $\tilde{\Psi}(\alpha,r_\alpha,0)=P^0(r_\alpha)$. 
If we prove that there exist values of $\alpha$ such that $\partial_{\mu}\tilde{\Psi}(\alpha,r_\alpha,0)<0$ then, in a neighborhood of $\mu=0$, 
we would have 
$$\Psi^{\mu}_{\varepsilon}(\alpha,r_\alpha)<\tilde{\Psi}(\alpha,r_\alpha,\mu)<\tilde{\Psi}(\alpha,r_\alpha,0)=P^0(r_\alpha),
$$
which would be a contradiction.
Since there exist $\varepsilon_0>0$ and  $\alpha_0 > 0$
such that for all $(\varepsilon,\alpha)\in [0,\varepsilon_0]\times[\alpha_0,1]$, the function
\begin{align*}
    \partial_r\partial_{\mu} \tilde{\Psi}(\alpha,r,0)&=\frac{4(1-\varepsilon)}{r}\int_0^r\tau^3e^{P^0(\tau)}d\tau+4\varepsilon r\int_0^r\tau e^{P^0(\tau)}d\tau-\alpha r\\
    &=\frac{2(1-\varepsilon)}{r}\left(\ln(1+r^2)+\frac{1}{1+r^2}-1\right)+4\varepsilon r\left(1-\frac{1}{1+r^2}\right)-\alpha r
\end{align*}
is less than $0$ for all $r>0$, we deduce that
$$
\partial_{\mu}\tilde{\Psi}(\alpha,r_\alpha,0)=\int_0^{r_\alpha}\partial_r\partial_{\mu} \tilde{\Psi}(\alpha,r,0)dr<0,
$$
which leads to the desired contradiction and finishes the proof of \eqref{CompP7}. We may establish~\eqref{CompdP7} in a very similar way, but using the expressions for $x\cdot\nabla P^0$ and $x\cdot\nabla P$. 
We finally prove \eqref{eq:estimQunifChap7} by taking the exponential of the estimate~\eqref{CompP7}.
\end{proof}

\begin{lem}\label{lem:UnifBddQP} There exist some constants $C_i >0$, $i=0,\ldots,3$, $\eps_0 > 0$ and $\vartheta\in(0,1)$, such that for any $\mu \in (0,\infty)$ and any $\eps \in (0,\varepsilon_0]$, 
there hold
\beqn\label{eq:estimUnifQ}
 0 \le Q(x) \le C_0 \, e^{-\mu \vartheta|x|^2/2} \langle x \rangle^{-4},
\eeqn
\beqn\label{eq:NablaPuniformBound7}
 \sup_{x \in \R^2} (\frac{1}{|x|} + \langle x \rangle) \, |\nabla P(x)| \le C_1,
\eeqn
and
\beqn\label{eq:borneDeltaVeps8}
|\Delta P | \le C_2 \mu \eps + C_3 \langle x \rangle^{-1}. 
 \eeqn

\end{lem}

\begin{proof}[Proof of Lemma \ref{lem:UnifBddQP}]
    Consider the values of $\varepsilon_0> 0$ and $\alpha_0 > 0$ given in   Proposition \ref{UniConChap7}, so that 
 from its proof,   $\vartheta:=1-\alpha_0 \in (0,1)$ is independent from $\mu$ and $\varepsilon$. 
 The estimate \eqref{eq:estimUnifQ} is then nothing but  \eqref{eq:estimQunifChap7}.
    Computing the explicit expression for $\nabla P$ gives
    $$
    \nabla P=-e^{-\mu\varepsilon\frac{|x|^2}{2}}\frac{x}{|x|^2}\int_0^{|x|}Q(r)e^{\mu\varepsilon\frac{r^2}{2}}r dr,
    $$
   and  hence
    $$
    |\nabla P|\le \frac{1}{|x|}\int_0^{|x|}Q(r)e^{\mu\varepsilon\frac{r^2}{2}}dr\le \frac{1}{|x|}\int_0^{|x|}\langle r\rangle^{-4}e^{-\mu(\vartheta-\varepsilon)\frac{r^2}{2}}r dr.
    $$
     Taking $\varepsilon$ small enough, we deduce
    $$
    |\nabla P|\le \frac{1}{|x|}\int_0^{|x|}\langle r\rangle^{-4}r dr\le |x|\langle x\rangle^{-2},
    $$
    which directly implies \eqref{eq:NablaPuniformBound7}.
    Finally, writing 
    $$
    \Delta P=-\mu\varepsilon x\cdot \nabla P-Q,
    $$
      we conclude to \eqref{eq:borneDeltaVeps8} thanks to \eqref{CompdP7} and \eqref{eq:estimUnifQ}.
\end{proof}

\begin{lem}\label{lem:eps0}
There exist some constants $\vartheta\in (0,1)$, $C_i>0$, $i=1,\ldots,4$, such that for any $\mu \in (0,\infty)$, any $\eps \in (0,\varepsilon_0]$ and any $x\in\R^2$, there holds
\begin{align}
    |\nabla P_\eps^\mu -\nabla P_0^{\mu}| &\le  \sqrt{\mu\varepsilon} C_1,
    \label{dPmedPm07}\\
        |\Delta P_\eps^\mu - \Delta P_0^{\mu}| &\le \mu\varepsilon C_2,
    \label{PmePm07}\\
    |Q_\eps^\mu -Q_0^{\mu}| &\le \mu\varepsilon C_3 e^{-\frac{\vartheta\mu|x|^2}{2}},
    \label{QmeQm07}\\
    |\nabla Q_\eps^\mu -\nabla Q_0^{\mu}| &\le \mu\varepsilon C_4 e^{-\frac{\vartheta\mu|x|^2}{2}}.\label{dQmedQm07}
\end{align}
\end{lem}

\begin{proof}[Proof of Lemma \ref{lem:eps0}]
We recall that in radial variables, we have the expressions
\begin{align*}
    P^{\mu}_0(r)=&-8\int_0^r\frac{1}{\rho}\int_0^{\rho}e^{P^{\mu}_0-\mu\frac{\tau^2}{2}}\tau d\tau d\rho,\\
    P^{\mu}_{\varepsilon}(r)=&-8\int_0^r\frac{e^{-\mu\varepsilon\frac{\rho^2}{2}}}{\rho}\int_0^{\rho}e^{P^{\mu}_{\varepsilon}-\mu(1-\varepsilon)\frac{\tau^2}{2}}\tau d\tau d\rho
\end{align*}
which imply
\begin{align*}
    P^{\mu}_{\varepsilon}-P^{\mu}_{0}=&\left(P^{\mu}_{\varepsilon}-\int_0^r (P^{\mu}_{\varepsilon})'e^{\mu\varepsilon\frac{\tau^2}{2}}d\tau\right)+\left(\int_0^r (P^{\mu}_{\varepsilon})'e^{\mu\varepsilon\frac{\tau^2}{2}}d\tau-P^{\mu}_0\right)\\
    =&-\int_0^r\frac{e^{-\mu\varepsilon\frac{\rho^2}{2}}-1}{\rho}\int_0^{\rho}Q^{\mu}_{\eps}(\tau)e^{\mu\varepsilon\frac{\tau^2}{2}}\tau d\tau d\rho\\
    &-8\int_0^{r}\frac{1}{\rho}\int_0^{\rho}(e^{P^{\mu}_{\eps}-\mu(1-\varepsilon)\frac{\tau^2}{2}}-e^{P^{\mu}_{0}-\mu\frac{\tau^2}{2}})\tau d\tau.
\end{align*}
Directly from Proposition \ref{UniConChap7}, we know that $\int_0^{\rho}Q^{\mu}_{\eps}(\tau)e^{\mu\varepsilon\frac{\tau^2}{2}}\tau d\tau\le \int_0^{\rho}Q^0(\tau)\tau d\tau\le 8\pi$, and the mean value theorem gives us that
$$
e^{P^{\mu}_{\eps}-\mu(1-\varepsilon)\frac{\tau^2}{2}}-e^{P^{\mu}_0-\mu\frac{\tau^2}{2}}=(P^{\mu}_{\eps}-P^{\mu}_0+\mu\varepsilon \frac{\tau^2}{2})e^{h(\tau)}, 
$$
with $h(\tau)$ satisfying
$$
h(\tau)\le \max\{P^{\mu}_{\varepsilon}-\mu(1-\varepsilon)\frac{\tau^2}{2},P^{\mu}_{0}-\mu\frac{\tau^2}{2}\}\le P^0 -\frac{\vartheta\mu\tau^2}{2}.
$$
Thanks to Proposition \ref{UniConChap7}, we deduce 
$$
8e^{h(\tau)}\le Q^0(\tau)e^{-\frac{\vartheta\mu\tau^2}{2}}.
$$
Putting everything together, we get
\begin{align*}
    |P^{\mu}_{\eps}-P^{\mu}_0| &\le 8\pi\int_0^r\frac{1-e^{-\frac{\mu\varepsilon\rho^2}{2}}}{\rho}d\rho+\int_0^r\frac{1}{\rho}\int_0^{\rho}|P^{\mu}_{\eps}-P^{\mu}_0+\mu\varepsilon \frac{\tau^2}{2}|Q^0(\tau)\tau d\tau d\rho\\
    &\le  \mu\varepsilon K r^2+\int_0^r\int_0^{\rho}|P^{\mu}_{\eps}-P^{\mu}_0|Q^0(\tau)d\tau d\rho,
\end{align*}
where $K$ is a constant independent of $\mu$ and $\varepsilon$. Integrating by parts the integral term, we get
$$
|P^{\mu}_{\eps}-P^{\mu}_0|\le \mu\varepsilon K r^2+r\int_0^r|P^{\mu}_{\eps}-P^{\mu}_0|Q^0(\tau)d\tau,
$$
which thanks to Gr\"onwall's Lemma gives
$$
|P^{\mu}_{\eps}-P^{\mu}_0|\le\mu\varepsilon C_0|x|^2.
$$
This estimate together with a similar manipulation on the gradients of $P^{\mu}_{\varepsilon}$ and $P^{\mu}_{0}$ gives
 $$
|\nabla P^{\mu}_{\eps}-\nabla P^{\mu}_0|\le\mu\varepsilon C_0|x|,
$$
which interpolated with \eqref{eq:NablaPuniformBound7} gives \eqref{dPmedPm07}.\\
On the other hand, we have
\bean
   |Q^{\mu}_{\eps}-Q^{\mu}_0| &=&8|e^{P^{\mu}_{\varepsilon}-\mu(1-\varepsilon)\frac{r^2}{2}}-e^{P^{\mu}_0-\mu\frac{r^2}{2}}|\\
   &=&|P^{\mu}_{\varepsilon}-P^{\mu}_0+\mu\varepsilon \frac{r^2}{2}|8e^{h(r)} \\
  & \le& \mu\varepsilon C_2 |r|^2Q^0(r)e^{-\frac{\vartheta\mu r^2}{2}},
  \eean
  which is nothing but \eqref{QmeQm07}. 
Repeating the same process for $\nabla(Q-Q_{\mu})$ gives \eqref{dQmedQm07}.

Finally, using the equations for $P^{\mu}_{\eps}$ and $P^{\mu}_0$, we have
     $$
     \Delta(P^{\mu}_{\eps}-P^{\mu}_0)=-(Q^{\mu}_{\eps}-Q^{\mu}_{0})-\mu\varepsilon x\cdot \nabla P^{\mu}_{\eps}.
     $$
     We conclude to  \eqref{PmePm07} thanks to \eqref{QmeQm07} for the first term and thanks to \eqref{CompdP7} for the second one. 
\end{proof}

\section{Functional inequalities}
\label{sec:Fineq}

We gather in this section some functional inequalities that we shall use through the paper. First, we provide some estimates over the solution for the Poisson problem   in the spirit of 
 \cite[Lemma~B.2]{MR3615547}.
 
\begin{lem}\label{lem:borneKgHs} 
There holds 
\beqn\label{eq:KgH1}
\| \nabla \kappa * g \|_{\dot H^1} 
\lesssim  \| g \|_{ L^2} , \quad \forall g \in L^2.
 \eeqn
Moreover, for any $\sigma \in (0,1)$ there holds
\beqn\label{eq:KgHs}
\| \nabla \kappa * g \|_{\dot H^\sigma} 
\lesssim  \| g \|_{L^1 \cap L^2}   , \quad \forall g \in L^1 \cap L^2.
 \eeqn
Finally, for any $k > 2$, there holds
\beqn\label{eq:BddK*fL2-new}
 \| \nabla \kappa * g \|_{L^2} \lesssim  \| g \|_{L^2_{k}}, \quad \forall \, g \in L^2_{k,0}.
\eeqn

\end{lem}

\begin{proof}[Proof of Lemma~\ref{lem:borneKgHs}]
From the very definition
$
- \Delta\kappa*g = g, 
$
we have $\FF(\kappa * g) (\xi) = - |\xi|^{-2} \hat g$, or equivalently  $|\xi|^{\sigma} |\widehat{(\nabla \kappa * g)} (\xi)| = |\xi|^{\sigma-1} |\widehat g(\xi)|$ for any $\sigma \in[0,1]$. Together with the  Plancherel identity, we have 
$$
\| \partial^2_{ij} (\kappa * g) \|_{L^2}  
=  \| \xi_i \xi_j |\xi|^{-2} \widehat g \|_{L^2} \le  \|   \widehat g \|_{L^2} = \| g \|_{L^2}, 
$$
 what establishes \eqref{eq:KgH1}.

We now consider $\sigma \in (0,1)$ and we write 
\bean
\| \nabla \kappa * g \|_{\dot H^\sigma}^2 
&=& \int_{|\xi| \le 1} |\xi|^{2\sigma-2} |\widehat g|^2 +  \int_{|\xi| \ge 1} |\xi|^{2\sigma-2}  |\widehat g|^2 
\\
&\le& \| \widehat g \|_{L^\infty}^2 \int_{|\xi| \le 1} |\xi|^{2\sigma-2} +  \int  |\widehat g|^2 
\\
&\lesssim& \|g \|_{L^1}^2 + \| g \|_{L^2}^2,
\eean  
which completes the proof of \eqref{eq:KgHs}.

We finally  prove \eqref{eq:BddK*fL2-new}. We similarly have
$$
\| \nabla \kappa * g \|_{L^2}^2 
= \int \mathbf{1}_{|\xi| \le 1} \, \frac{|\widehat g(\xi)|^2}{|\xi|^2} \, \d \xi + \int \mathbf{1}_{|\xi| > 1} \, \frac{|\widehat g(\xi)|^2}{|\xi|^2} \, \d \xi =: I_1+I_2.
$$
For the second term we have
$$
I_2 \le  \int |\widehat g(\xi)|^2 \, \d \xi =  \| g \|^2_{L^2}.
$$
For the first term, using that  $\widehat g(0)=0$ because $\lla g \rra=0$, we have 
$$
\widehat g(\xi) = \xi \cdot \int_0^1 D_\xi \widehat g (\theta \xi) \, \d \theta,
$$
and we thus obtain
$$
\begin{aligned}
I_1 
\le  \left( \sup_{|\xi| \le 1} |D_\xi \widehat g (\xi) |^2 \right)\int \mathbf{1}_{|\xi| \le 1} \, \d \xi 
\lesssim \| \widehat{x g} \|_{L^\infty}^2
\lesssim \| x g \|_{L^1}^2 \lesssim \| g \|_{L^2_k}^2, 
\end{aligned}
$$
by using some classical and elementary Fourier identity and estimate as well as the continuous embedding $L^2_k (\R^2) \subset L^1_1 (\R^2)$.
\end{proof}

 We give now a possible variante of \eqref{eq:BddK*fL2-new}.

\begin{lem}\label{lem:borneKgLp7}  
For $p > 2$, $2 \le q \le p$ and   $k>2-2/q$, we have  
\beqn\label{eq:KgLp7}
\| \nabla \kappa * g \|_{L^p} \lesssim  \| g \|_{L^q_{k}}, \quad \forall \, g \in L^q_k.
\eeqn
\end{lem}

We emphasize that the same estimate is not true for $p=2$.   From the proof of \eqref{eq:BddK*fL2-new}, we indeed have 
$$
\| \nabla \kappa * g \|_{L^2} = \| |\xi|^{-1} \widehat g \|_{L^2}, 
$$
and the RHS term is infinite if $\widehat g(0) \not = 0$ or equivalently,  if the mass of $g$ does not vanish, whatever is its $L^q_k$ norm.

\begin{proof}[Proof of Lemma~\ref{lem:borneKgLp7}]
We split $|\nabla \kappa| := K_1 + K_2$, with 
$$
K_1 := \frac{1}{|x|} \mathbf{1}_{|x| \le 1} \in L^{r_1}, \, \forall \, r_1 < 2, \quad K_2 := \frac{1}{|x|} \mathbf{1}_{|x| \ge 1} \in L^{r_2}, \, \forall \, r_2 > 2,
$$
that we use with $r_1 :=  (1+\frac{1}{p}-\frac{1}{q})^{-1}{  \in(1,2)}$ and $r_2 := p   >2$. 
We then have 
\bean
\| \nabla \kappa *  f \|_{L^p} 
&\le&  \| K_1 *  f \|_{L^p} +  \| K_2 *  f \|_{L^p} \\
&\le&  \| K_1 \|_{L^{r_1}}  \|  f \|_{L^q} +  \| K_2 \|_{L^p} \| f \|_{L^1} \\
&\lesssim&   \|  f \|_{L^q_{k}}  ,
\eean
where we have used the convolution embeddings $L^{r_1} * L^q \subset L^p$ and $L^1 * L^p \subset L^p$ in the second line as well as the  H\"older inequality in the last line in order to prove $L^q_{k} \subset L^1$. 
\end{proof}

We finally recall two standard interpolation inequalities: a  particular case of the Gagliardo-Nirenberg interpolation Theorem in dimension 2 (see for instance \cite[pp~125]{nirenberg1959elliptic}) and an interpolation between homogeneous Sobolev spaces (see \cite[Proposition~1.32]{MR2768550}).

\begin{lem}\label{lem:GN}
\begin{enumerate}
\item The following Ladyzhenskaya's inequality holds
\beqn\label{eq:Lad_ineq}
\| f \|_{L^4} \lesssim \| f \|_{L^2}^{1/2} \| \nabla f \|_{L^2}^{1/2}, \quad \forall \, f \in H^1.
\eeqn

\item Consider real numbers $s_0 < s < s_1$, then
\beqn\label{eq:interpolation_dotHs}
\| f \|_{\dot H^s} \lesssim \| f \|_{\dot H^{{s_0}}}^{1-\theta} \| f \|_{\dot H^{s_1}}^{\theta}, \quad \forall \, f \in \dot H^{s_0} \cap \dot H^{s_1},
\eeqn
where $\theta = \frac{s-s_0}{s_1-s_0}$.

\end{enumerate}

\end{lem}


\section{Estimates for $\LL_{1,1}$}
\label{sec:EstimL11}
In this section, we establish some dissipativity estimates and related semigroup decay estimates successively on the operators $\LL_{1,1}$ and related operators.

\subsection{Dissipativity estimates related to $\LL_{1,1}$. }\label{sec:Beps}

In order to keep track of the $\eps \ge  0$ dependence, let us denote
$$
\Lambda_\eps := \LL_{1,1},
$$
where we recall that this one is defined by 
$$
\LL_{1,1} g := \Delta g + \Div (\mu x g - g \nabla P - Q \nabla \kappa*g ). 
$$
We start with a first fundamental dissipativity estimate.

\begin{lem}\label{lem:Lambda_eps_L2}
For any $k > 3$, there   exist some constants $\eps_0>0$, small enough, and  $C_0,\varrho_0>0$, large enough,  such that 
\begin{equation}\label{eq:Lambda_eps}
\la \Lambda_\eps g , g \ra_{L^2_k} 
\le - \mu (k-2)\| g \|_{L^2_k}^2 - \frac12 \| \nabla g \|_{L^2_k}^2 + C_0 \| g \|_{L^2(B_{\varrho_0})}^2, 
\end{equation}
for any   $\eps \in (0,\eps_0)$ and $g \in H^{{1}}_k$.
\end{lem}

\begin{proof}[Proof of Lemma~\ref{lem:Lambda_eps_L2}]
We briefly repeat the proof of \cite[Lemma 4.4]{MR3615547}.  
We compute
$$
\begin{aligned}
\la \Lambda_\eps g , g \ra_{L^2_k}
&= \int \Delta g  g \la x \ra^{2k} + \int \Div (\mu x g) g \la x \ra^{2k}
- \int \Div(g \nabla P) g \la x \ra^{2k}
- \int \Div (Q \nabla \kappa * g) g \la x \ra^{2k} \\
&= I_1+I_2+I_3+I_4,
\end{aligned}
$$
and estimate each term separately.
For the two first terms we have
$$
\begin{aligned}
I_1+I_2
&= - \int |\nabla g|^2 \la x \ra^{2k} + \int \psi_1 g^2 \la x \ra^{2k}
\end{aligned}
$$
where
$$
\begin{aligned}
\psi_1 
&= \frac{|\nabla \la x \ra^{k}|^2}{\la x \ra^{2k}} + \frac{\Delta \la x \ra^{k}}{\la x \ra^{k}} + \mu - \mu x \cdot \frac{\nabla \la x \ra^{k} }{\la x \ra^{k}} \\
&= k(2 k +  \mu) \la x \ra^{-2}  - \mu(k-1)  -2k \la x \ra^{-4}   .
\end{aligned}
$$
Moreover, for the third term we compute
$$
\begin{aligned}
I_3 
&= \int g \nabla P \cdot \nabla g \la x \ra^{2k} + 2\int g^2 \la x \ra^{2k} \nabla P \cdot \frac{\nabla \la x \ra^{k}}{\la x \ra^{k}} \\
&= -\frac{1}{2}\int  \Delta P  g^2 \la x \ra^{2k} + \int  \nabla P \cdot \frac{\nabla \la x \ra^{k}}{\la x \ra^{k}}  \, g^2 \la x \ra^{2k}   .
\end{aligned}
$$
Thanks to the uniform estimates \eqref{eq:NablaPuniformBound7} and \eqref{eq:borneDeltaVeps8} on  $P$, we observe that
$$
\left| \nabla P \cdot \frac{\nabla \la x \ra^{k}}{\la x \ra^{k}} \right| 
\le \| x\cdot\nabla P \|_{L^\infty} \la x \ra^{-1} \le C_1 \la x \ra^{-1}
$$
and 
$$
|\Delta P | \le C_2 \mu \eps + C_3 \langle x \rangle^{-1}
$$
for some constants $C_i > 0$, which imply 
$$
I_3 \le \frac{\eps \mu C_2}{2} \| g \|_{L^2_k}^2 + \left( C_1 + \frac{C_3}{2} \right) \| \la x \ra^{-\frac12} g \|_{L^2_k}^2.
$$
For the last term we write
$$
\begin{aligned}
I_4 
&= \int  Q (\nabla \kappa * g) \nabla g \la x \ra^{2k} +  2\int  Q (\nabla \kappa * g)  \frac{\nabla \la x \ra^{k}}{\la x \ra^{k}} \, g \la x \ra^{2k}.
\end{aligned}
$$
Since $\| Q \la x \ra^{2k} \|_{L^\infty} \le C_4$ and $\| Q \la x \ra^{k} \nabla \la x \ra^{k} \|_{L^\infty} \le C_4$ thanks to estimate \eqref{eq:estimUnifQ}, we obtain
$$
\begin{aligned}
I_4 
&\le C_4 \| \nabla \kappa * g \|_{L^2} \| \nabla g \|_{L^2}+  2C_4  \| \nabla \kappa * g \|_{L^2} \|  g \|_{L^2}\\
&\le C'_4 \| \nabla \kappa * g \|_{L^2}^2 + \frac12 \| \nabla g \|_{L^2}^2+  C'_4  \| g \|_{L^2}^2 \\
&\le C''_4 \| \la x \ra^{-1} g \|_{L^2_k}^2 + \frac12 \| \nabla g \|_{L^2}^2,
\end{aligned}
$$
where we have used Lemma~\ref{lem:borneKgHs} and Young's inequality. 
Gathering the previous estimates, we get 
$$
\begin{aligned}
\la \Lambda_\eps g , g \ra_{L^2_k} 
\le  - \frac12 \int |\nabla g|^2 \la x \ra^{2k} + \int \bar \psi_1 g^2 \la x \ra^{2k}
\end{aligned}
$$
with
$$
\begin{aligned}
\bar \psi_1 
&= - \mu \left( k-1 - \frac{\eps C_2}{2} \right) + \left( C''_4+  C_1 + \frac{C_3}{2}\right) \la x \ra^{-1} +  k(2k+\mu)   \la x \ra^{-2} - k^2 \la x \ra^{-4} \\
&\le - \mu \left( k-1 - \frac{\eps C_2}{2} \right) + C_5 \la x \ra^{-1} .\end{aligned}
$$
We remark that, for any $\varrho_0 \ge 1$, we have
$$
\la x \ra^{-1} \la x \ra^{2k} \le \varrho_0^{2k-1} \mathbf{1}_{\la x \ra \le \varrho_0} +  \frac{1}{\varrho_0} \la x \ra^{2k},
$$
thus we obtain
\begin{equation}\label{eq:Lambda_eps_gg}
\begin{aligned}
\la \Lambda_\eps g , g \ra_{L^2_k} 
\le  - \frac12 \| \nabla g \|_{L^2_k}^2 - \mu \left( k-1 - \frac{\eps C_2}{2} - \frac{C_5}{\mu \varrho_0} \right)  \| g \|_{L^2_k}^2 + C_0 \| g \|_{L^2(B_{\varrho_0})}^2
\end{aligned}
\end{equation}
where $C_0 = C_5 \varrho_0^{2k-1}$.
We therefore choose $\eps_0>0$ small enough such that $\eps_0 C_2 \le 1$ and $\varrho_0 \ge 1$ large enough such that $\frac{C_5}{\mu \varrho_0} \le 1/2$, which concludes the proof. 
\end{proof}

\subsection{Splitting of the operator $\LL_{1,1}$}
\label{subsec-splitL11}

We introduce the splitting
$$
\Lambda_\eps = \AA + \BB_\eps, \quad \AA := M \chi_{\varrho} , \quad \BB_\eps := \Lambda_\eps - \AA,
$$
with $\chi_{\varrho}(x) := \chi(x/\varrho)$, $\chi \in \DD(\R^2)$, ${\bf 1}_{B(0,1)} \le \chi \le {\bf 1}_{B(0,2)}$, and constants $M,\varrho>0$. We immediately deduce from Lemma~\ref{lem:Lambda_eps_L2} that $\BB_\eps$
is dissipative, more precisely: 

\begin{cor}\label{cor:Lambda_eps_L2}
For any $k > 3$, any $\eps \in (0,\eps_0)$ 
and any constants $M \ge C_0$ and $\varrho \ge \varrho_0$, there holds
\begin{equation}\label{eq:B_eps}
\la \BB_\eps g , g \ra_{L^2_k} 
\le -\mu(k-2) \| g \|_{L^2_k}^2 - \frac12\| \nabla g \|_{L^2_k}^2
\le - \lambda \| g \|_{L^2_k}^2 - \sigma \| g \|_{H^1_k}^2  
\end{equation}
for any $0 \le \lambda < \mu(k-2)$ with $\sigma = \min (1/2, \mu-\lambda)$, and where $\eps_0,C_0,\varrho_0>0$ are taken from Lemma~\ref{lem:Lambda_eps_L2}.
\end{cor}

\begin{rem}
We shall fix hereafter the parameters $M \ge C_0$ and $\varrho \ge \varrho_0$ in the definition of $\BB_\eps$ such that Corollary~\ref{cor:Lambda_eps_L2} holds.
\end{rem}

In order to work at the level of the semigroup, we reformulate \eqref{eq:B_eps} in the following way, recalling that we define $H^{-1}_k$ through the dual norm
$$
\| R \|_{H^{-1}_k} := \sup_{\| g \|_{H^1_k} \le 1} \la R,g \ra_{L^2_k}.
$$

\begin{lem}\label{lem:SBB_eps_L2}
For any $k > 3$, $\eps \in (0,\eps_0)$, 
$M \ge C_0$ and $\varrho \ge \varrho_0$, there holds
\begin{enumerate}

\item For all $0 \le \lambda <  \mu (k-2)$ and all $g \in L^2_k$, we have
$$
 \| \mathbf{e}_\lambda S_{\BB_\eps} (\cdot) g \|_{L^\infty_t L^2_k} 
 + \|\mathbf{e}_\lambda S_{\BB_\eps}(\cdot) g \|_{L^2_t H^1_k} \lesssim  \| g \|_{L^2_k}.
$$

\item For all $0 \le \lambda < \mu (k-2)$ and all $\mathbf{e}_{\lambda} R \in L^2_t H^{-1}_{k}$, we have
$$
\begin{aligned}
\left\| \mathbf{e}_\lambda (S_{\BB_\eps} * R)  \right\|_{L^\infty_t L^2_k} + \left\| \mathbf{e}_\lambda (S_{\BB_\eps} * R)  \right\|_{L^2_t H^1_{k}}  \lesssim  \| \mathbf{e}_\lambda R\|_{L^2_t H^{-1}_k}.
\end{aligned} 
$$
\end{enumerate}
\end{lem}

\begin{proof}[Proof of Lemma~\ref{lem:SBB_eps_L2}]
Let $0 \le \lambda < \mu (k-2)$. 
For $g \in L^2_k$, we first   consider $f :=  \mathbf{e}_\lambda S_{\BB_\eps} (\cdot) g$ the  solution to the evolution equation 
$$
\partial_t f = \BB_\eps f + \lambda f, \quad f(0) = g.
$$
Because of \eqref{eq:B_eps}, we have 
$$
\frac12 \frac{d}{dt} \| f \|_{L^2_k}^2 
= \la \BB_\eps f , f \ra_{L^2_k} + \lambda\| f \|_{L^2_k}^2 \le  - \sigma \|  f \|_{H^1_k}^2  
$$
from which we deduce (1) thanks to the Gr\"onwall's lemma. 

\smallskip
For $R$ such that $\mathbf{e}_{\lambda} R \in L^2_t H^{-1}_{k}$,  we next consider $f :=  \mathbf{e}_\lambda (S_{\BB_\eps} * R) $ the solution to the evolution equation 
$$
\partial_t f = \BB_\eps f + \lambda f + \mathbf{e}_\lambda R, \quad f(0) = 0.
$$
Because of \eqref{eq:B_eps} and the Young inequality, we have 
\bean
\frac12\frac{d}{dt} \| f \|_{L^2_k}^2 
&=& \la \BB_\eps f , f \ra_{L^2_k} + \lambda \| f \|_{L^2_k}^2 + \la \mathbf{e}_\lambda R , f \ra_{L^2_k} 
\\
&\le&  -\sigma \| f \|_{H^1_k}^2 +  \|  \mathbf{e}_\lambda R \|_{H^{-1}_k}   \|  f \|_{H^1_k}  
\\
&\le&  - \frac{\sigma}{2} \| f \|_{H^1_k}^2  + C \|  \mathbf{e}_\lambda R\|_{H^{-1}_k}^2,
\eean
for some constant $C = C(\mu,\lambda) > 0$. We deduce (2) thanks to Gr\"onwall's lemma again. 
\end{proof}

\subsection{Spectral analysis of $\LL_{1,1}$} 
We deduce a nice localization of the spectrum of $\LL_{1,1}$ from the previous estimates and a perturbation argument. 
Let us denote by  $\Lambda_0$ the linearized operator of the parabolic-elliptic Keller-Segel equation which is given by
$$
\Lambda_0 g = \Delta g + \Div (\mu x g - g \nabla P_0 - Q_0 \nabla \kappa*g ), 	
$$
where $(Q_0,P_0)$ is a solution to \eqref{eq:KS_ssQPChap7} with $\eps=0$.
From \cite{MR3196188,MR3466844}, we know that for $k>3$ and  $0<\lambda < \mu$, there exists a constant $C = C(\lambda,\mu,k) \ge 1$ such that  
$$
\| S_{\Lambda_0} (t) f\|_{L^2_k} \le C e^{-\lambda t} \| f \|_{L^2_k}, \quad \forall \, f \in L^2_{k,0}, 
$$
and the spectrum  verifies
\beqn\label{eq:SigmaLambda0}
\Sigma(\Lambda_0) \cap \Delta_{-\mu} = \{0 \}
\eeqn
where $\Delta_{-\mu}:=\{z\in\mathbb{C}:\Re e z >-\mu \}$.

By a perturbation argument similar to the one used in \cite{MMcmp} (see also \cite{MR3452720,MR1335452}), we are able to obtain a similar picture for
the operator $\LL_{1,1} = \Lambda_\eps$. 

\begin{prop}\label{prop:Lambda_eps}
Let $k>3$. For any $0<\lambda < \mu$, there is $\eps_*>0$ small enough, such that the operator $\Lambda_\eps$ on $L^2_k$ satisfies
$$
\Sigma(\Lambda_\eps) \cap \Delta_{-\mu} = \{0 \}, \quad \forall \, \eps \in (0,\eps^*).
$$

\end{prop}

\begin{proof}[Proof of Proposition~\ref{prop:Lambda_eps}]
 We split the proof into several steps.

\medskip
\noindent
{\sl Step 1. } We claim that 
$$
\UU_\eps(z) := \RR_{\BB_\eps}(z) - \RR_{\Lambda_0}(z) \AA \RR_{\BB_\eps}(z) 
$$
is uniformly bounded in  $\BBB(L^2_k)$ and $\BBB(H^{-1}_k,H^1_k)$ for any $z \in \Omega := \Delta_{-\lambda} \backslash B(0,r/2)$ any $\eps \ge 0$ and $0 < r < \lambda < \mu$. 
On the one hand, $\RR_{\BB_\eps}(z) \in \BBB(L^2_k)$ is just an immediate consequence of the growth estimate on $S_{\BB_\eps}$ established in Lemma~\ref{lem:SBB_eps_L2}-(1). 
For proving $\RR_{\BB_\eps}(z) \in \BBB(H^{-1}_k,H^1_k)$,  we consider first $g \in L^2_k$, $z \in \Delta_{-\mu}$ and we define $f := \RR_{\BB_\eps}(z) g$, so that 
$(z-\BB_\eps) f = g$. Using \eqref{eq:B_eps} and the fact that $\mu (k-2) \ge \mu$, we deduce 
$$
\frac12 \| \nabla f \|_{L^2_k}^2 + (\Re e z + \mu)  \|   f \|_{L^2_k}^2 \le \langle (z-\BB_\eps) f , f \rangle_{L^2_k} = \langle f, g \rangle_{L^2_k} \le  \|   f \|_{H^1_k}  \| g \|_{H^{-1}_k} 
 $$
 and thus 
 \beqn
 \label{eq:RBepsH-1H1}   \| \nabla f \|_{L^2_k} \le \max\bigl( 2, \mu^{-1} \bigr)  \| g \|_{H^{-1}_k}. 
\eeqn
 By a density argument, the same holds for any $g \in H^{-1}_k$. From \eqref{eq:SigmaLambda0}, we also have $\RR_{\Lambda_0}(z) \in \BBB(L^2_k)$ uniformly bounded in  $\BBB(L^2_k)$  for any $z \in \Omega := \Delta_{-\lambda} \backslash B(0,r/2)$. Moreover, the proof of the bound in $\BBB(L^2_k,H^1_k)$ is exactly the same as for $\RR_{\BB_\eps}(z)$.
  Indeed arguing as in Lemma~\ref{lem:Lambda_eps_L2}, we first obtain
$$
\la \Lambda_0 f , f \ra_{L^2_k} \le - \mu \| f \|_{L^2_k}^2 - \frac12 \| \nabla f \|_{L^2_k}^2 + C_0 \| f \|_{L^2(B_{\varrho_0})}^2.
$$
Defining $f := \RR_{\Lambda_0}(z) g$, we deduce
$$
(\Re e z + \mu) \| f \|_{L^2_k}^2 + \frac12 \| \nabla f \|_{L^2_k}^2 - C_0 \| f \|_{L^2(B_{\varrho_0})}^2 \le \langle (z-\Lambda_0) f , f \rangle_{L^2_k} = \la g, f \ra _{L^2_k} \le C \| g \|_{L^2_k}^2.
$$

\medskip
\noindent
{\sl Step 2. } We claim that the family of operators $(\Lambda_\eps)$ converges 
%
in the sense
$$
 \| \Lambda_\eps - \Lambda_0  \|_{\BBB(H^1_{k},L^2_k)} \le \eta_1(\eps) \to0.
$$
We may indeed write 
\bean
(\Lambda_\eps - \Lambda_0)g
&=& ( \BB_\eps - \BB_0)g
\\
&=& - \Div(g \nabla (P_\eps-P_0)) - \Div(  (Q_\eps-Q_0)  \nabla \kappa * g)
\\
&=& -\nabla g \cdot \nabla (P_\eps-P_0) + g \Delta (P_\eps-P_0) 
\\
&&\qquad +  \nabla (Q_\eps-Q_0) \cdot  \nabla \kappa * g  -  (Q_\eps-Q_0)   g,
\eean
so that 
\bean
\| (\Lambda_\eps - \Lambda_0)g \|_{L^2_k} 
&\le& \| \nabla (P_\eps-P_0) \|_{L^\infty} \| \nabla g \|_{L^2_k} + \| \Delta (P_\eps-P_0)  \|_{L^\infty} \|  g \|_{L^2_k} 
\\
&&\qquad +  \| \la x \ra^k \nabla (Q_\eps-Q_0) \|_{L^\infty} \|  g \|_{L^2_{1+0}} + \| Q_\eps-Q_0  \|_{L^\infty} \|  g \|_{L^2_k}.
\eean
We immediately conclude since we are able to prove (see Lemma~\ref{lem:eps0})
$$
 \nabla (P_\eps-P_0) \to 0, \quad  \Delta (P_\eps-P_0)  \to 0, \quad \la x \ra^k \nabla (Q_\eps-Q_0)  \to 0, \quad Q_\eps-Q_0  \to 0
$$
uniformly in $L^\infty(\R^2)$. 

\medskip
\noindent
{\sl Step 3. } We claim that   $\Sigma(\Lambda_\eps) \cap \Delta_a \subset B(0,r/2)$ for any $\eps \in (0,\eps_0)$, choosing $\eps_0 > 0$ small enough. 
On the one hand, we write the two resolvent equations
\bean
\RR_{\Lambda_\eps} &=& \RR_{\BB_\eps} -  \RR_{\Lambda_\eps} \AA \RR_{\BB_\eps} ,
\\
\RR_{\Lambda_\eps} &=& \RR_{\Lambda_0} -  \RR_{\Lambda_\eps} (\Lambda_\eps - \Lambda_0) \RR_{\Lambda_0}, 
\eean
from what we deduce 
\bean
\RR_{\Lambda_\eps} =  \RR_{\BB_\eps} -   \RR_{\Lambda_0} \AA \RR_{\BB_\eps}  + \RR_{\Lambda_\eps} (\Lambda_\eps - \Lambda_0) \RR_{\Lambda_0} \AA \RR_{\BB_\eps} ,
\eean
or equivalently 
$$
\RR_{\Lambda_\eps} (I + \KK_\eps) = \UU_\eps,
$$
with 
$$
\KK_\eps :=  (\Lambda_0 - \Lambda_\eps) \RR_{\Lambda_0} \AA \RR_{\BB_\eps} .
 $$
 On the other hand, from {\it Step 1}, we have $\RR_{\Lambda_0}(z)\AA \RR_{\BB_\eps} (z)$ is  bounded in   $\BBB(L^2_k,H^1_k)$ uniformly in $z \in \Omega := \Delta_{-\lambda} \backslash B(0,r/2)$
 and $\Lambda_0 - \Lambda_\eps$ is small in $\BBB(H^1_k,L^2_k)$ for $\eps > 0$ small, so that both estimates together imply
$$
\sup_{z \in \Delta_{-\lambda} \backslash B(0,r/2)} \|\KK_\eps (z)\|_{\BBB(L^2)} < 1, 
$$
for any $0< r < \lambda < \mu$ and  $\eps \in (0,\eps_0)$, with $\eps_0 = \eps_0(r,\lambda) > 0$ small enough. This implies that $I + \KK_\eps$ is invertible on 
$\Omega := \Delta_{-\lambda} \backslash B(0,r/2)$ so that 
$$
\RR_{\Lambda_\eps}= \UU_\eps  (I + \KK_\eps) ^{-1}
$$
is bounded on $\Omega$, which ends the proof. 

 \medskip
\noindent
{\sl Step 4. } We define now 
$$
\Pi_{\eps}  :=  \frac{i}{2\pi} \int_{\Gamma} \RR_{\Lambda_\eps}(z) \, dz, \quad \Gamma  := \{ z \in \C; |z| = r \},
$$
the Dunford projector on the eigenspace associated to eigenvalues of  $\Lambda_{\eps}$ which belong to the ball  $B(0,r)$. 
We write 
\bean
\Pi_\eps 
&=& \frac{i}{2 \pi} \int_\Gamma \UU_\eps  \,   (I+ \KK_{\eps})^{-1} \, dz
\\
&=& \frac{i}{2 \pi} \int_\Gamma  \RR_{\BB_\eps }  
 \, \{  I - \KK_{\eps}  \, (I+\KK_{\eps} )^{-1} \} \, dz
  - \frac{i}{2 \pi} \int_\Gamma     \, \RR_{\Lambda_0} \, \AA \, \RR_{\BB_\eps}  \, (I+\KK_{\eps} )^{-1}  \, dz
\\
&=& - \frac{i}{2 \pi} \int_\Gamma  \RR_{\BB_\eps }    \KK_{\eps}  \, (I+\KK_{\eps} )^{-1}  \, dz
  - \frac{i}{2 \pi} \int_\Gamma     \, \RR_{\Lambda_0} \, \AA \, \RR_{\BB_\eps}  \, (I+\KK_{\eps} )^{-1}  \, dz ,
\eean
and
\bean
\Pi_0
&=& \frac{i}{2 \pi} \int_\Gamma \{ \RR_{\BB_0} -  \RR_{\Lambda_0} \AA  \RR_{\BB_0} \} \, dz
\\
&=& -  \frac{i}{2 \pi} \int_\Gamma  \RR_{\Lambda_0} \AA  \RR_{\BB_0} \{ (I + \KK_\eps)^{-1}  + \KK_\eps  (I + \KK_\eps)^{-1} \}  \, dz .
\eean
We deduce 
\bean
\Pi_\eps - \Pi_0 
&=&    \frac{i}{2 \pi} \int_\Gamma (\RR_{\Lambda_0} \AA  \RR_{\BB_0}- \RR_{\BB_\eps } )  \KK_{\eps}  \, (I+\KK_{\eps} )^{-1}  \, dz
\\
&& + 
\frac{i}{2 \pi} \int_\Gamma  \RR_{\Lambda_0} \AA \{ \RR_{\BB_0} - \RR_{\BB_\eps} \} (I + \KK_\eps)^{-1} \, dz
\\
&=&   -  \frac{i}{2 \pi} \int_\Gamma \UU_\eps  \KK_{\eps}  \, (I+\KK_{\eps} )^{-1}  \, dz
\\
&& + 
\frac{i}{2 \pi} \int_\Gamma  \RR_{\Lambda_0} \AA  \RR_{\BB_0} \{ \BB_0 - \BB_\eps \}  \RR_{\BB_\eps}   (I + \KK_\eps)^{-1} \, dz.
\eean
We conclude that $\| \Pi_\eps - \Pi_0 \|_{\BBB(L^2)} = \OO(\eps) < 1$ for $\eps > 0$ small enough by taking advantage of the estimates established in {\it Step 1} and {\it Step 2}.
By classical operator theory (see for instance the arguments presented in \cite{MR1335452} in order to prove \cite[Chap 1, (4.43)]{MR1335452}) one deduces that $\hbox{\rm dim}\, \Pi_\eps = \hbox{\rm dim}\, \Pi_0 = 1$.
On the other hand, at first glance we have $\Lambda^*_\eps 1 = 0$ and $1 \in (L^2_k)'$ so that $0 \in \Sigma (\Lambda_\eps^*) = \Sigma (\Lambda_\eps)$, and $0$ is the only spectral value of $\Lambda_\eps$ in the ball $B(0,r)$. 
\end{proof}

\subsection{Semigroup decay estimates for $\LL_{1,1}$} 
We are now able to deduce a nice semigroup decay estimate on $S_{\LL_{1,1}}$ from the previous estimates on the resolvent.

\begin{prop}\label{prop:SLambda_eps}
With the notation of Proposition~\ref{prop:Lambda_eps}, for $k>3$, all $0<\lambda < \mu$ and all $\eps \in (0, \eps_*)$ there holds, for any $g \in L^{2}_{k,0}$,
$$
\| S_{\LL_{1,1}} (t) g\|_{L^2_k} \lesssim e^{-\lambda t} \| g \|_{L^2_k}.
$$
\end{prop}

\begin{proof}[Proof of Proposition~\ref{prop:SLambda_eps}] 
It  is a consequence of Proposition~\ref{prop:Lambda_eps} and of the splitting structure of the operator $\Lambda_\eps$. 
More precisely, we may for instance apply the quantitative mapping theorem \cite[Theorem~2.1]{MR3489637}, where it is worth emphasizing that $\RR_{\BB_\eps}(z) : L^2_k \to H^1_k \subset D(\Lambda_\eps^{1/2})$ with uniformly bound in $z \in \Delta_{-\lambda}$, which is a strong enough information in order 
to establish \cite[(2.23)]{MR3489637} without checking  \cite[{\bf (H2)}]{MR3489637}. Alternatively, one can use the  Gearhart-Pr\"uss-Greiner theorem  \cite{MR0461206,MR0743749,MR0839450} in order to get the same conclusion. 
\end{proof}

Thanks to the previous estimate for $S_{\LL_{1,1}}$ and the estimates for $S_{\BB_\eps}$ in Lemma~\ref{lem:SBB_eps_L2}, we are able to deduce semigroup estimates for $S_{\LL_{1,1}}$ (in Propositions~\ref{prop:SLambda_eps_L2} below) similar to those satisfied by $S_{\BB_\eps}$.

We start  observing that, thanks to Duhamel's formula, we have
\beqn\label{eq:Duhamel-1}
S_{\LL_{1,1}} = S_{\BB_\eps} + S_{\BB_\eps} \AA *  S_{\LL_{1,1}}
\quad\text{and}\quad
S_{\LL_{1,1}} = S_{\BB_\eps} + S_{\LL_{1,1}} * \AA S_{\BB_\eps}.
\eeqn
Denoting $\Pi^\perp g = g-\Pi g$, where $\Pi$ is the projection onto $\operatorname{Ker} (\Lambda_\eps)$, 
we obtain
\beqn\label{eq:Duhamel-2}
S_{\LL_{1,1}}\Pi^\perp = S_{\BB_\eps}\Pi^\perp + (S_{\BB_\eps} \AA *  S_{\LL_{1,1}}\Pi^\perp )
\quad\text{and}\quad
S_{\LL_{1,1}} \Pi^\perp = \Pi^\perp S_{\BB_\eps} + (S_{\LL_{1,1}}\Pi^\perp * \AA S_{\BB_\eps}).
\eeqn
Using that $S_{\LL_{1,1}} \Pi^\perp = \Pi^\perp S_{\LL_{1,1}} $, and iterating the formulas  yields
$$
S_{\LL_{1,1}}\Pi^\perp = S_{\BB_\eps}\Pi^\perp + S_{\BB_\eps} \AA * \Pi^\perp S_{\BB_\eps} + S_{\BB_\eps} \AA * S_{\LL_{1,1}}\Pi^\perp * \AA S_{\BB_\eps}
$$
and
\beqn\label{eq:Duhamel-3}
S_{\LL_{1,1}}\Pi^\perp = \Pi^\perp S_{\BB_\eps} + S_{\BB_\eps}\Pi^\perp * \AA S_{\BB_\eps} + S_{\BB_\eps} \AA * S_{\LL_{1,1}}\Pi^\perp * \AA S_{\BB_\eps}.
\eeqn

\medskip
We deduce by combining some results of section~\ref{subsec-splitL11} and the above Proposition~\ref{prop:SLambda_eps}, some additional estimates on the semigroup $S_{\LL_{1,1}}$. 

\begin{prop}\label{prop:SLambda_eps_L2}
Let $0 \le \lambda < \mu$ and $k>3$. There is $\eps_*>0$ small enough such that for any   $\eps \in (0,   \eps_*  )$ the following holds:
\begin{enumerate}

\item For all $g \in L^2_{k,0}$ we have
$$
\| \mathbf{e}_\lambda  S_{\LL_{1,1}} (\cdot) g \|_{L^\infty_t L^2_k} +   \| \mathbf{e}_\lambda  S_{\LL_{1,1}}(\cdot) g \|_{L^2_t H^1_k} \lesssim  \| g \|_{L^2_k}.
$$

\item For all $\mathbf{e}_{\lambda} R \in L^2_t H^{-1}_{k}$ with $\Pi R = 0$, we have
$$
\begin{aligned}
\left\| \mathbf{e}_\lambda  (S_{\LL_{1,1}} * R) \right\|_{L^\infty_t L^2_k} + \left\| \mathbf{e}_\lambda  (S_{\LL_{1,1}} * R) \right\|_{L^2_t H^1_k}  \lesssim  \| \mathbf{e}_\lambda  R \|_{L^2_t H^{-1}_k}.
\end{aligned} 
$$
\end{enumerate}

\end{prop}

\begin{proof}[Proof of Proposition~\ref{prop:SLambda_eps_L2}] 
\textit{Proof of (1).}
Remark that $g = \Pi^\perp g$, since $g \in L^2_{k,0}$.
The first estimate is nothing but Proposition~\ref{prop:SLambda_eps}. 
In particular, we deduce from this one that 
\beqn\label{eq:SLambda_eps_inL1}
\|  \mathbf{e}_{\lambda} S_{\LL_{1,1}}\Pi^\perp   \|_{L^1_t \BBB(L^2_k)} \lesssim \|  \mathbf{e}_{\lambda'} S_{\LL_{1,1}} \Pi^\perp \|_{L^\infty_t \BBB(L^2_k)} \lesssim 1,
\eeqn
by choosing $\lambda < \lambda' < 1$.
On the other hand, 
thanks to Lemma~\ref{lem:SBB_eps_L2}-(1) and the first estimate, we have
$$
 \| \mathbf{e}_{\lambda} S_{\BB_\eps}(\cdot) g \|_{L^2_t H^1_k} \lesssim \| g \|_{L^2_k}
$$
and
$$
\begin{aligned}
 \| \mathbf{e}_{\lambda} (S_{\BB_\eps} * \AA S_{\LL_{1,1}})(\cdot) g \|_{L^2_t H^1_k}
&\lesssim \| \mathbf{e}_{\lambda} \AA S_{\LL_{1,1}} (\cdot) g \|_{L^2_t H^{-1}_k} \\
&\lesssim \| \mathbf{e}_{\lambda} S_{\LL_{1,1}} (\cdot) g \|_{L^2_t L^2_k}\\
&\lesssim \| g \|_{L^2_k},
\end{aligned}
$$
where we have used that $\AA \in \BBB(L^2_k)$, from which together with the first identity in \eqref{eq:Duhamel-1}, we immediately obtain the second estimate.

\medskip\noindent
\textit{Proof of (2).}
Remarking  that $R(t) = \Pi^\perp R(t)$, for all $t \ge 0$, and using the second identity in \eqref{eq:Duhamel-2}, we may write
\beqn\label{eq:Duhamel-4}
S_{\LL_{1,1}} * R = S_{\LL_{1,1}}\Pi^\perp * R = \Pi^\perp (S_{\BB_\eps} * R) + S_{\LL_{1,1}}\Pi^\perp * \AA S_{\BB_\eps} * R,
\eeqn
and thus
$$
\mathbf{e}_{\lambda} (S_{\LL_{1,1}} * R) 
= \mathbf{e}_{\lambda} \Pi^\perp (S_{\BB_\eps} * R) + (\mathbf{e}_{\lambda} S_{\LL_{1,1}} \Pi^\perp ) * \AA[\mathbf{e}_{\lambda} (S_{\BB_\eps} * R)].
$$
We deduce
\bean
\| \mathbf{e}_{\lambda} (S_{\LL_{1,1}} * R) \|_{L^\infty_t L^2_k}
&\le&\| \mathbf{e}_{\lambda} (S_{\BB_\eps} * R) \|_{L^\infty_t L^2_k}
+ \| (\mathbf{e}_{\lambda} S_{\LL_{1,1}} \Pi^\perp) * \AA[\mathbf{e}_{\lambda} (S_{\BB_\eps} * R)] \|_{L^\infty_t L^2_k} 
\\
&\le & \| \mathbf{e}_{\lambda} (S_{\BB_\eps} * R) \|_{L^\infty_t L^2_k}\\
&&
+ \| \mathbf{e}_{\lambda} S_{\LL_{1,1}} \Pi^\perp \|_{L^1_t ( \BBB(L^2_k) )} \| \AA \|_{\BBB(L^2_k)}\| \mathbf{e}_{\lambda} (S_{\BB_\eps} * R) \|_{L^\infty_t L^2_k} 
\\
&\lesssim &\| \mathbf{e}_{\lambda} R \|_{L^2_t H^{-1}_k},
\eean
where we have used Lemma~\ref{lem:SBB_eps_L2}--(2) and \eqref{eq:SLambda_eps_inL1}  in last line. We have established the first estimate in {\it (2)}. 

\smallskip

For the second term, using \eqref{eq:Duhamel-3}, we may write  
$$
S_{\LL_{1,1}} * R 
= \Pi^\perp (S_{\BB_\eps} * R) + S_{\BB_\eps} \Pi^\perp *\AA S_{\BB_\eps} * R + S_{\BB_\eps} \AA * S_{\LL_{1,1}} \Pi^\perp * \AA  S_{\BB_\eps} * R, 
$$
and thus
$$
\mathbf{e}_{\lambda} (S_{\LL_{1,1}} * R) 
= \Pi^\perp \mathbf{e}_{\lambda} (S_{\BB_\eps} * R) 
+ \mathbf{e}_{\lambda} [S_{\BB_\eps} * (\Pi^\perp  \AA  S_{\BB_\eps} * R) ]
+ \mathbf{e}_{\lambda} [ (S_{\BB_\eps} \AA) * ( S_{\LL_{1,1}} \Pi^\perp *\AA S_{\BB_\eps} * R)].
$$
We now estimate each term separately.
From Lemma~\ref{lem:SBB_eps_L2}-(2), we have
$$
\| \mathbf{e}_{\lambda} (S_{\BB_\eps} * R)  \|_{L^2_t H^1_k} 
\lesssim \| \mathbf{e}_{\lambda}  R  \|_{L^2_t H^{-1}_k}.
$$
From Lemma~\ref{lem:SBB_eps_L2}-(2) again, we also have
$$
\begin{aligned}
\| \mathbf{e}_{\lambda} [S_{\BB_\eps} * (\Pi^\perp \AA  S_{\BB_\eps} * R) ]
  \|_{L^2_t H^1_k} 
&\lesssim \| \mathbf{e}_{\lambda}  (\Pi^\perp \AA  S_{\BB_\eps} * R)  \|_{L^2_t H^{-1}_k}
\\
&\lesssim \| \mathbf{e}_{\lambda}  (\Pi^\perp \AA  S_{\BB_\eps} * R)  \|_{L^2_t L^{2}_k} \\
&\lesssim \| \mathbf{e}_{\lambda}  (S_{\BB_\eps} * R)  \|_{L^2_t L^2_k} \\
&\lesssim \| \mathbf{e}_{\lambda}  R  \|_{L^2_t H^{-1}_k}.
\end{aligned}
$$
We finally have
$$
\begin{aligned}
&\| \mathbf{e}_{\lambda} [(S_{\BB_\eps} \AA) * ( S_{\LL_{1,1}} \Pi^\perp * \AA S_{\BB_\eps} * R)]
  \|_{L^2_t H^1_k} \\
&\quad\lesssim \|   ( \mathbf{e}_{\lambda}S_{\LL_{1,1}} \Pi^\perp ) *[ \AA \mathbf{e}_{\lambda}( S_{\BB_\eps} * R)]  \|_{L^2_t H^{-1}_k} \\
&\quad\lesssim \|   ( \mathbf{e}_{\lambda}S_{\LL_{1,1}} \Pi^\perp ) *[ \AA \mathbf{e}_{\lambda}( S_{\BB_\eps} * R)]  \|_{L^2_t L^{2}_k} \\
&\quad\lesssim \|    \mathbf{e}_{\lambda}S_{\LL_{1,1}} \Pi^\perp \|_{L^1_t(\BBB(L^2_k))} \| \AA \|_{\BBB(L^2_k)} \| [ \mathbf{e}_{\lambda}( S_{\BB_\eps} * R)]  \|_{L^2_t L^{2}_k} \\
&\quad\lesssim\| \mathbf{e}_{\lambda}  R  \|_{L^2_t H^{-1}_k},
\end{aligned}
$$
where we have used Lemma~\ref{lem:SBB_eps_L2}-(2) in the second line as well as Proposition~\ref{prop:SLambda_eps} and Lemma~\ref{lem:SBB_eps_L2}-(2) in the last line. 
Putting the three last estimates together, we conclude to the second estimate in (2). 
\end{proof}

\section{Estimates for $\LL_{2,2}$}
\label{sec:EstimL12}

In this section, we are concerned with establishing estimates on $S_{\LL_{2,2}}$, where we recall that $\LL_{2,2}$ is defined  by 
$$
\LL_{2,2} w = \frac{1}{\eps} \Delta w + \mu x \cdot \nabla w + \nabla \kappa *[Q \nabla w] .
$$
Before presenting our key estimate, we make several observations. 

\smallskip
On the one hand, we have  
\bean
\nabla \LL_{2,2} w 
&=&\frac{1}{\eps} \Delta (\nabla w) + \mu x \cdot \nabla (\nabla w) + \mu \nabla w  + \nabla^2 \kappa *[Q \nabla w],
\eean
where we denote
$$
(x \cdot \nabla \Phi)_i = x_\ell \partial_\ell \Phi_i, \quad
(\nabla^2 \kappa * \Phi)_i = \partial_{i \ell} \kappa * \Phi_\ell, 
$$
for any vector $\Phi$. A straightforward computation then  gives 
\bean
\la \nabla \LL_{2,2} w, \nabla w \ra_{L^2} 
&=&  \int (\dfrac1\eps  \Delta \nabla w + \mu   \nabla w + \mu x \cdot \nabla^2 w  +   \nabla^2 \kappa * [ Q \nabla w] )\nabla w
\\
&=& - \dfrac1\eps  \int | \nabla^2 w|^2 - \int Q|\nabla w|^2, 
\eean 
where we have performed two integrations by parts for the last term and we have used the identity $\Delta \kappa = - \delta$. That nice estimate tells us that $\LL_{2,2}$ is dissipative in $\dot H^1$, but it does not provide a spectral gap. Namely, recalling that
$$
\la \nabla w , \nabla z \ra_{L^2} = \frac{1}{(2 \pi)^2} \la w , z \ra_{\dot H^1}
\quad\text{and}\quad
\| \nabla^2 w \|_{L^2}^2 = \frac{1}{(2 \pi)^2} \| w \|_{\dot H^2},
$$
we have established that:
\begin{lem}\label{lem:LL2_dotH1}
For any $w \in \dot H^{1} \cap \dot H^2$, there holds
$$
\begin{aligned}
\la \LL_{2,2} w , w \ra_{\dot H^1} 
= - \frac{1}{\eps} \| w \|_{\dot H^2}^2 -  (2\pi)^2\| Q^{1/2} \nabla w \|_{L^2}^2.
\end{aligned}
$$
\end{lem}

We now establish that  $\LL_{2,2}$ has a stronger dissipativity property  in $\dot H^s$ for $s\in(0,1)$.

\begin{lem}\label{lem:LL2_dotHs}
For any $s \in (0,1)$ and $\vartheta < \mu(1-s)$, there is $\eps > 0$ small enough such that  
$$
\la \LL_{2,2} w , w \ra_{\dot H^{s}} 
\le - \frac{1}{2\eps} \| w \|_{\dot H^{s+1}}^2 - \vartheta  \| w \|_{\dot H^s}^2 , 
$$
for any $w \in H^{s+1}$.
\end{lem}

\begin{proof}[Proof of Lemma~\ref{lem:LL2_dotHs}]  
We split the operator $\LL_{2,2} = \LL_{2,2,1} + \LL_{2,2,2}$ with
$$
\LL_{2,2,1} w := \frac{1}{\eps} \Delta w + \mu x \cdot \nabla w
$$
and
$$
\LL_{2,2,2} w := \nabla \kappa * [Q \nabla w].
$$

For the first term we write 
\bean
\la \LL_{2,2,1} w , w \ra_{\dot H^{s}} 
&:=& \left\la  \frac{1}{\eps} \Delta w + \mu x \cdot \nabla w , w \right\ra_{\dot H^{s}} 
\\
&=& \int |\xi|^{2s} \left\{ - \frac{1}{\eps} |\xi|^2 \widehat w - \mu \Div (\xi \widehat  w) \right\} \overline{\widehat w}  
\\
&=&- \frac{1}{\eps} \int |\xi|^{2 + 2s} |\widehat w|^2 -   \mu \int \left[ 2  |\xi|^{2s} - \frac12 \Div (\xi |\xi|^{2s}) \right] | \widehat w|^2, 
\\
&=&- \frac{1}{\eps} \| w \|_{\dot H^{1+s}}^2 -   \mu (1-s)  \| w \|_{\dot H^{s}}^2  
\eean
which gives the contribution of the first part of the operator $\LL_{2,2}$  in $\dot H^{s}$.

For the second part of the  operator $\LL_{2,2}$  in $\dot H^{s}$, we shall take advantage of the fact that, roughtly speaking, the operator $\nabla \kappa * $ behaves as taking a primitive. More precisely, we have
\bean
\la \LL_{2,2,2} w , w \ra_{\dot H^{s}} 
&:=& \la   \nabla \kappa *[Q \nabla w]   , w \ra_{\dot H^{s}} 
\\
&\leq&  \|  \nabla \kappa *[Q \nabla w] \|_{ \dot H^{s}}  \| w \|_{ \dot H^{s}}  
\\
&\lesssim&  \|   Q \nabla w \|_{L^1 \cap L^2}  \| w \|_{ \dot H^{s}}  
\\
&\lesssim&  \|    \nabla w \|_{L^2}  \| w \|_{ \dot H^{s}}  
\\
&\lesssim&    \| w \|_{ \dot H^{s}}^{1+s}   \| w \|_{ \dot H^{s+1}}^{1-s} ,  
\eean
where we have used Cauchy-Schwarz inequality in the second line, estimate~ \eqref{eq:KgHs} of Lemma~\ref{lem:borneKgHs} in the third one, the exponential decay of $Q$ in Lemma~\ref{lem:UnifBddQP} in the fourth one, and finally the interpolation inequality~\eqref{eq:interpolation_dotHs} in the last one.

We can conclude that $\LL_{2,2}$ is dissipative in $\dot H^s$. Fix $0 \le \vartheta < \mu(1-s)$, then gathering previous estimates it follows that
$$
\begin{aligned}
\la \LL_{2,2} w , w \ra_{\dot H^{s}} 
\le - \frac{1}{\eps} \| w \|_{\dot H^{s+1}}^2 
- \mu(1-s)  \| w \|_{\dot H^s}^2 
+ C \| w \|_{ \dot H^{s}}^{1+s}   \| w \|_{ \dot H^{s+1}}^{1-s},
\end{aligned}
$$
and we conclude by applying Young's inequality to the last term and choosing $\eps>0$ small enough.
\end{proof}

As a consequence of previous estimates, we obtain the following decay and regularization estimates for the semigroup $S_{\LL_{2,2}}$ in $\dot H^s \cap \dot H^1$. 
 Roughly speaking, in that space the semigroup $S_{\LL_{2,2}}$ inherit both the good dissipativity property of  $S_{\LL_{2,2}}$ in $\dot H^s$ and the good regularity 
property of $S_{\LL_{2,2}}$ in $\dot H^1$, that last point behing crucial for dealing with the nonlinear term coming from the first equation.

\begin{lem}\label{lem:S_LL22_L2}
Let $s \in (0,1)$ and $0 \le \vartheta < \mu(1-s)$. There is $\eps_2>0$ small enough such that for any $\eps \in (0, \eps_2)$ the following holds:
\begin{enumerate}

\item For all $w \in \dot H^s \cap \dot H^1$, we have
$$
 \| \mathbf{e}_\vartheta S_{\LL_{2,2}} (\cdot) w \|_{L^\infty_t (\dot H^s \cap \dot H^1)} 
+ \|\mathbf{e}_\vartheta  S_{\LL_{2,2}}(\cdot) w \|_{L^2_t \dot H^s}
+ \frac{1}{\sqrt{\eps}}\|\mathbf{e}_\vartheta  S_{\LL_{2,2}}(\cdot) w \|_{L^2_t (\dot H^{s+1} \cap \dot H^2)} \lesssim  \| w \|_{\dot H^s \cap \dot H^1}.
$$

\item For all $\mathbf{e}_{\vartheta} S \in L^2_t (\dot H^s \cap \dot H^1)$, we have
$$
\begin{aligned}
\left\| \mathbf{e}_\vartheta (S_{\LL_{2,2}} * S)  \right\|_{L^\infty_t (\dot H^s \cap \dot H^1)} 
&+ \left\| \mathbf{e}_\vartheta  (S_{\LL_{2,2}} * S)  \right\|_{L^2_t \dot H^s }
+ \frac{1}{\sqrt{\eps}} \left\| \mathbf{e}_\vartheta  (S_{\LL_{2,2}} * S)  \right\|_{L^2_t (\dot H^{s+1} \cap \dot H^2)}  \\
&\lesssim    \| \mathbf{e}_\vartheta S \|_{L^2_t (\dot H^s \cap \dot H^1)}.
\end{aligned} 
$$

\end{enumerate}
\end{lem}

\begin{proof}[Proof of Lemma~\ref{lem:S_LL22_L2}]  
Let us denote $\phi =  \mathbf{e}_\vartheta S_{\LL_{2,2}} (\cdot) w$ the solution to the evolution equation 
$$
\partial_t \phi = \LL_{2,2} \phi + \vartheta \phi, \quad \phi(0) = w.
$$
Fix $\vartheta < \vartheta' < \mu(1-s)$.
Thanks to Lemma~\ref{lem:LL2_dotH1} and Lemma~\ref{lem:LL2_dotHs}, we have
$$
\frac{1}{2} \frac{d}{dt} \left\{ \| \phi \|_{\dot H^s}^2 + \| \phi \|_{\dot H^1}^2 \right\}
\le 
- \frac{1}{2 \eps} \left\{ \| \phi \|_{\dot H^{s+1}}^2 + \| \phi \|_{\dot H^2}^2 \right\}
-\left(\vartheta' - \vartheta \right) \| \phi \|_{\dot H^s}^2
+ \vartheta \| \phi \|_{\dot H^1}^2. 
$$
Thanks to the interpolation inequality~\eqref{eq:interpolation_dotHs} together with Young's inequality, we observe that
$$
\vartheta\| \phi \|_{\dot H^1}^2 \le \frac{(\vartheta'-\vartheta)}{2} \| \phi \|_{\dot H^s}^2 + C \| \phi \|_{\dot H^{s+1}}^2
$$
for some $C>0$. Gathering previous estimates and taking $\eps>0$ small enough yield
$$
\frac{1}{2} \frac{d}{dt} \left\{ \| \phi \|_{\dot H^s}^2 + \| \phi \|_{\dot H^1}^2 \right\}
\le 
- \frac{1}{4 \eps} \left\{ \| \phi \|_{\dot H^{s+1}}^2 + \| \phi \|_{\dot H^2}^2 \right\}
-\frac{(\vartheta' - \vartheta)}{2} \| \phi \|_{\dot H^s}^2,
$$
from which (1) follows by Gr\"onwall's lemma.

\medskip

We now consider $\phi =  \mathbf{e}_\vartheta (S_{\LL_{2,2}} * S) $ the  solution to the evolution equation 
$$
\partial_t \phi = \LL_{2,2} \phi + \vartheta \phi + \mathbf{e}_\vartheta S, \quad \phi(0) = 0.
$$
Arguing as above we have, with $\vartheta < \vartheta' < \mu(1-s)$,
$$
\begin{aligned}
\frac{1}{2} \frac{d}{dt} \left\{ \| \phi \|_{\dot H^s}^2 + \| \phi \|_{\dot H^1}^2 \right\}
&\le 
- \frac{1}{4 \eps} \left\{ \| \phi \|_{\dot H^{s+1}}^2 + \| \phi \|_{\dot H^2}^2 \right\}
-\frac{(\vartheta' - \vartheta)}{2} \| \phi \|_{\dot H^s}^2 \\
&\quad
+ \la \phi , \mathbf{e}_\vartheta S \ra_{\dot H^s}
+ \la \phi , \mathbf{e}_\vartheta S \ra_{\dot H^1}.
\end{aligned}
$$
Writing
$$
\la \phi , \mathbf{e}_\vartheta S \ra_{\dot H^s} 
\le  \frac{(\vartheta' - \vartheta)}{8}\| \phi \|_{\dot H^s}^2 + C\| \mathbf{e}_\vartheta S \|_{\dot H^s}^2
$$
for some $C>0$, and
$$
\la \phi , \mathbf{e}_\vartheta S \ra_{\dot H^1}
\le  \| \phi \|_{\dot H^1}^2 + \| \mathbf{e}_\vartheta S \|_{\dot H^1}^2,
$$
we then use \eqref{eq:interpolation_dotHs} again to get
$$
\| \phi \|_{\dot H^1}^2 \le \frac{(\vartheta'-\vartheta)}{8} \| \phi \|_{\dot H^s}^2 + C' \| \phi \|_{\dot H^{s+1}}^2
$$
for some $C'>0$. Taking $\eps>0$ small enough then yields
$$
\begin{aligned}
\frac{1}{2} \frac{d}{dt} \left\{ \| \phi \|_{\dot H^s}^2 + \| \phi \|_{\dot H^1}^2 \right\}
&\le 
- \frac{1}{8 \eps} \left\{ \| \phi \|_{\dot H^{s+1}}^2 + \| \phi \|_{\dot H^2}^2 \right\}
-\frac{(\vartheta' - \vartheta)}{4} \| \phi \|_{\dot H^s}^2 
\\
&\quad
+ C \left( \| \mathbf{e}_\vartheta S \|_{\dot H^s}^2 + \| \mathbf{e}_\vartheta S \|_{\dot H^1}^2 \right),
\end{aligned}
$$
and we conclude to (2) applying Gr\"onwall's lemma again.
\end{proof}

\section{Semigroup estimates for the linearized system}
\label{sec:EstimSGL}

We start with some estimates on the out of the diagonal operators $\LL_{1,2}$ and $\LL_{2,1}$ which we recall are defined in \eqref{eq:KS_defLij} by 
$$ 
\LL_{1,2} w = - \Div (Q \nabla w), \quad
\LL_{2,1} g =  g + \nabla \kappa * \left[ g \nabla P + Q \nabla \kappa*g \right].
$$

\begin{lem}\label{lem:L12&L21}
Let $s\in(0,1)$. For $k > 3$, there holds
$$
\| \LL_{1,2} w \|_{H^{-1}_k}  
\lesssim  \|   w \|_{\dot H^s}^{1-\theta} \|   w \|_{\dot H^2}^{\theta}, \quad \forall w \in \dot H^s \cap \dot H^2, 
$$
with $\theta= (1-s)/(2-s)$, and 
$$
 \| \LL_{2,1} g \|_{\dot H^s} +  \| \LL_{2,1} g \|_{\dot H^1}     
\lesssim \| g \|_{H^1_k}, \quad \forall \, g \in H^1_k.
$$

\end{lem}

\begin{proof}[Proof of Lemma~\ref{lem:L12&L21}]  
For the first estimate, we write 
$$
\begin{aligned}
\| \Div( Q  \nabla w ) \|_{H^{-1}_k}
&= \sup_{ \| \psi \|_{H^1_k}\le 1 } \int \Div( Q  \nabla w ) \psi \langle x \rangle^{2k}
\\
&= \sup_{ \| \psi \|_{H^1_k}\le 1 } \int   Q  \nabla w  \left( \langle x \rangle^{2k} \nabla \psi +  \psi \nabla \langle x \rangle^{2k} \right)
\\
&\lesssim  \| \nabla w \|_{L^2} , 
\end{aligned}
$$
where we have used   the Cauchy-Schwarz inequality and the uniform bound 
of $Q$ (see Lemma~\ref{lem:UnifBddQP}) in the last line.
Using the interpolation  inequality~\eqref{eq:interpolation_dotHs}, we then obtain
$$
\begin{aligned}
\| \Div( Q  \nabla w ) \|_{H^{-1}_k}
&\lesssim \| w \|_{\dot H^s}^{{1/(2-s)}} \| w \|_{\dot H^2}^{{(1-s)/(2-s)}},
\end{aligned}
$$
which is nothing but the first claimed estimate. 

\medskip
For the second estimate, we write  
$$
\begin{aligned}
\| \LL_{2,1} g \|_{\dot H^s \cap \dot H^1} \le \| g \|_{\dot H^s \cap \dot H^1} + \| \nabla \kappa * [g \nabla P] \|_{\dot H^s \cap \dot H^1} + \| \nabla \kappa * [Q \nabla \kappa *g] \|_{\dot H^s \cap \dot H^1},
\end{aligned}
$$
and we estimate the three terms at the RHS separately. 
On the one hand, we have
$$
\begin{aligned}
\| \nabla \kappa * [g \nabla P] \|_{\dot H^s \cap \dot H^1}
&\lesssim \| g \nabla P \|_{L^1 \cap L^2} \\
&\lesssim \| g \|_{L^2_k}
\end{aligned}
$$
where we have used Lemma~\ref{lem:borneKgHs} in the first line, and Lemma~\ref{lem:UnifBddQP} in the second one together with the embedding $L^2_k(\R^2) \subset L^1(\R^2)$. Similarly, we have for some $p>2$
$$
\begin{aligned}
\| \nabla \kappa * [ Q \nabla \kappa * g] \|_{\dot H^s \cap \dot H^1}
&\lesssim \| Q \nabla \kappa * g\|_{L^1 \cap L^2} \\
&\lesssim \| \nabla \kappa * g \|_{L^p}\\
&\lesssim \| g \|_{L^2_k}
\end{aligned}
$$
by using successively Lemma~\ref{lem:borneKgHs}, Lemma~\ref{lem:UnifBddQP} and Lemma~\ref{lem:borneKgLp7}. 
We conclude by putting together the two previous estimates and observing that $\| g \|_{\dot H^s \cap \dot H^1} \lesssim \| g \|_{H^1_k}$.
\end{proof}

As a consequence of Proposition~\ref{prop:SLambda_eps_L2},  Lemma~\ref{lem:S_LL22_L2} and Lemma~\ref{lem:L12&L21}, recalling the definition of the spaces $\XX$ and $\YY$ in \eqref{eq:def:XX&YY},
we obtain the following semigroup estimates on the linearized problem. 

\begin{prop}\label{prop:S_LL_eps}
Let $s \in(0,1)$, $0 \le \lambda < \mu(1-s)$ and $k>3$. There is $\eps_*>0$ small enough such that for any $\eps \in (0, \eps_*)$ there holds:
\begin{enumerate}

\item For any $(g_0,w_0) \in L^2_{k,0} \times (\dot H^s \cap \dot H^1)$ we have
$$
\| \mathbf{e}_{\lambda} S_{\LL} (\cdot) (g_0,w_0) \|_{L^\infty_t (\XX)} + \| \mathbf{e}_{\lambda} S_{\LL} (\cdot) (g_0,w_0) \|_{L^2_t (\YY)} \lesssim \| (g_0,w_0) \|_{\XX}.
$$

\item For any $\mathbf{e}_{\lambda} \RR = \mathbf{e}_{\lambda} (\RR_1,\RR_2) \in L^2_t ( H^{-1}_{k} \times (\dot H^s \cap \dot H^1))$ with $\Pi \RR_1=0$ we have
$$
\| \mathbf{e}_{\lambda} ( S_{\LL} * \RR) \|_{L^\infty_t (\XX)} 
+ \| \mathbf{e}_{\lambda} ( S_{\LL} * \RR) \|_{L^2_t (\YY)}
\lesssim \| \mathbf{e}_{\lambda} \RR \|_{L^2_t (H^{-1}_{k} \times  (\dot H^s \cap \dot H^1) )}.
$$

\end{enumerate}

\end{prop}

\begin{proof}[Proof of Proposition~\ref{prop:S_LL_eps}]  
We split the proof into two steps.

\medskip\noindent
\textit{Proof of (1).}
Let us denote
$$
(g(t) , w(t)) = S_{\LL}(t) (g_0,w_0), 
$$
so that
$$
g(t) = S_{\LL_{1,1}}(t) g_0 + (S_{\LL_{1,1}} * \LL_{1,2} w )(t)
$$
and
$$
w(t) = S_{\LL_{2,2}}(t) w_0 + ( S_{\LL_{2,2}} * \LL_{2,1} g)(t) .
$$
We observe that $\lla \LL_{1,2} w \rra =0$ so that $\Pi (\LL_{1,2} w ) =0$, and we can then  hereafter apply the results of Proposition~\ref{prop:SLambda_eps_L2} to $S_{\LL_{1,1}} * \LL_{1,2} w$.
From Proposition~\ref{prop:SLambda_eps_L2}--(1), we have
$$
\begin{aligned}
\| \mathbf{e}_{\lambda}  S_{\LL_{1,1}}(\cdot) g_0 \|_{L^\infty_t L^2_k } + \| \mathbf{e}_{\lambda}  S_{\LL_{1,1}}(\cdot) g_0 \|_{L^2_t H^1_k}
&\lesssim \| g_0 \|_{L^2_k}.
\end{aligned}
$$
On the other hand, from Proposition~\ref{prop:SLambda_eps_L2}--(2) , we have
$$
\begin{aligned}
\| \mathbf{e}_{\lambda} ( S_{\LL_{1,1}} * \LL_{1,2} w) \|_{L^\infty_t L^2_k} + \| \mathbf{e}_{\lambda} ( S_{\LL_{1,1}} * \LL_{1,2} w) \|_{L^2_t H^1_k}
&\lesssim \| \mathbf{e}_{\lambda}  \LL_{1,2} w \|_{L^2_t H^{-1}_k} \\
&\lesssim \| \mathbf{e}_{\lambda} w \|_{L^2_t \dot H^s}^{1-\theta} \| \mathbf{e}_{\lambda} w \|_{L^2_t \dot H^2}^{\theta} 
\end{aligned}
$$
with ${\theta = (1-s)/(2-s)}$, where we have used the first estimate in Lemma~\ref{lem:L12&L21}.
We have thus established
\begin{equation}\label{eq:g}
\| \mathbf{e}_{\lambda} g \|_{L^\infty_t L^2_k} + \| \mathbf{e}_{\lambda} g \|_{L^2_t H^1_k} 
\le C_1 \| g_0 \|_{L^2_k} + C_2 \| \mathbf{e}_{\lambda} w \|_{L^2_t \dot H^s}^{1-\theta} \| \mathbf{e}_{\lambda} w \|_{L^2_t \dot H^2}^{\theta},
\end{equation}
for some constants $C_1, C_2>0$.

\medskip
We now come to the estimate of $w$.  We recall that from Lemma~\ref{lem:S_LL22_L2}--(1), we have
$$
\| \mathbf{e}_{\lambda} S_{\LL_{2,2}} (\cdot) w_0 \|_{L^\infty_t (\dot H^s \cap \dot H^1) } + \| \mathbf{e}_{\lambda} S_{\LL_{2,2}} (\cdot) w_0 \|_{L^2_t \dot H^s} + \frac{1}{\sqrt{\eps}} \| \mathbf{e}_{\lambda} S_{\LL_{2,2}} (\cdot) w_0 \|_{L^2_t \dot H^2} \lesssim \| w_0 \|_{\dot H^s \cap \dot H^1}.
$$
Moreover from Lemma~\ref{lem:S_LL22_L2}--(2) we have
$$
\begin{aligned}
\| \mathbf{e}_{\lambda} ( S_{\LL_{2,2}} * \LL_{2,1} g) \|_{L^\infty_t (\dot H^s \cap \dot H^1)}
+ \| \mathbf{e}_{\lambda} ( S_{\LL_{2,2}} * \LL_{2,1} g) \|_{L^2_t \dot H^s}
&+ \frac{1}{\sqrt{\eps}}\| \mathbf{e}_{\lambda} ( S_{\LL_{2,2}} * \LL_{2,1} g) \|_{L^2_t \dot H^2} \\
&\lesssim \| \mathbf{e}_{\lambda}  \LL_{2,1} g \|_{L^2_t (\dot H^s \cap \dot H^1)} \\
&\lesssim \| \mathbf{e}_{\lambda}  g \|_{L^2_t H^1_k} ,
\end{aligned}
$$
 where we have used the second estimate in Lemma~\ref{lem:L12&L21} for obtaining the last inequality. 
 The two last estimates together, we have thus established 
\begin{equation}\label{eq:w}
\| \mathbf{e}_{\lambda} w \|_{L^\infty_t (\dot H^s \cap \dot H^1)} 
+ \| \mathbf{e}_{\lambda} w \|_{L^2_t \dot H^s}
+ \frac{1}{\sqrt \eps}\| \mathbf{e}_{\lambda} w \|_{L^2_t \dot H^2} 
\le C_3 \| w_0 \|_{\dot H^s \cap \dot H^1} + C_4 \| \mathbf{e}_{\lambda} g \|_{L^2_t H^1_k},
\end{equation}
for some constants $C_3,C_4>0$.

\medskip
Coming back to \eqref{eq:g} and using Young's inequality, we deduce   that for any $\beta>0$, there is some $C_\beta >0$ such that
$$
\| \mathbf{e}_{\lambda} g \|_{L^\infty_t L^2_k} + \| \mathbf{e}_{\lambda} g \|_{L^2_t H^1_k} 
\le C_1 \| g_0 \|_{L^2_k} + \beta \| \mathbf{e}_{\lambda} w \|_{L^2_t \dot H^s} + C_\beta \| \mathbf{e}_{\lambda} w \|_{L^2_t \dot H^2}.
$$
Combining that last estimate with \eqref{eq:w} yields
$$
\begin{aligned}
\| \mathbf{e}_{\lambda} g \|_{L^\infty_t L^2_k} + \| \mathbf{e}_{\lambda} g \|_{L^2_t H^1_k} 
&\le C_1 \| g_0 \|_{L^2_k} 
+ \beta C_3 \| w_0 \|_{\dot H^s \cap \dot H^1} + \beta C_4 \| \mathbf{e}_{\lambda} g \|_{L^2_t H^1_k}  \\
&\quad
+ \sqrt{\eps} C_\beta C_3 \| w_0 \|_{\dot H^s \cap \dot H^1} +  \sqrt{\eps} C_\beta C_4 \| \mathbf{e}_{\lambda} g \|_{L^2_t H^1_k}.
\end{aligned}
$$
Choosing first $\beta>0$ small enough and then $\eps>0$ small enough, we obtain 
$$
\begin{aligned}
\| \mathbf{e}_{\lambda} g \|_{L^\infty_t L^2_k} + \| \mathbf{e}_{\lambda} g \|_{L^2_t H^1_k} 
&\le C_5 \| g_0 \|_{L^2_k} 
+  C_6 \| w_0 \|_{\dot H^s \cap \dot H^1} 
\end{aligned}
$$
for some constants $C_5,C_6>0$. We then conclude part (1) by gathering this last estimate with \eqref{eq:w}.

\medskip\noindent
\textit{Step 2.} 
Let us denote
$
(G , W)(t) = (S_{\LL}* \RR)(t) ,
$
so that
$$
G(t) = (S_{\LL_{1,1}} * \LL_{1,2} W )(t) +  (S_{\LL_{1,1}} * \RR_1 )(t)
$$
and
$$
W(t) = ( S_{\LL_{2,2}} * \LL_{2,1} G)(t) + ( S_{\LL_{2,2}} * \RR_2)(t) .
$$
Observing that $\Pi \RR_1 = \Pi (\LL_{1,2} W) = 0$, we may use Proposition~\ref{prop:SLambda_eps_L2}--(2) as well as additionally  
 Lemma~\ref{lem:L12&L21} for handling $S_{\LL_{1,1}} * \LL_{1,2} W$, and we obtain 
\begin{equation*}
\| \mathbf{e}_{\lambda} G \|_{L^\infty_t L^2_k} + \| \mathbf{e}_{\lambda} G \|_{L^2_t H^1_k} 
\le C_2 \| \mathbf{e}_{\lambda} W \|_{L^2_t \dot H^s}^{1-\theta} \| \mathbf{e}_{\lambda} W \|_{L^2_t \dot H^2}^{\theta} + C'_1 \| \mathbf{e}_{\lambda} \RR_1 \|_{L^2_t H^{-1}_k} ,
\end{equation*}
for some constants $C'_1, C_2>0$ and with ${ \theta = (1-s)/(2-s)}$. 

Similarly, using Lemma~\ref{lem:S_LL22_L2}--(2) and  additionally Lemma~\ref{lem:L12&L21} for handling  $S_{\LL_{2,2}} * \LL_{2,1} G$, we have 
\begin{equation*}
\| \mathbf{e}_{\lambda} W \|_{L^\infty_t (\dot H^s \cap \dot H^1)} 
+ \| \mathbf{e}_{\lambda} W \|_{L^2_t \dot H^s}
+ \frac{1}{\sqrt \eps}\| \mathbf{e}_{\lambda} W \|_{L^2_t \dot H^2} 
\le C_4 \| \mathbf{e}_{\lambda} G \|_{L^2_t H^1_k} + C'_3 \| \mathbf{e}_{\lambda} \RR_2 \|_{L^2_t (\dot H^s \cap \dot H^1)} ,
\end{equation*}
for some constants $C'_3,C_4>0$.
We can then conclude to (2) by arguing as in Step~1.
\end{proof}

\section{Proof of the nonlinear stability theorem}
\label{sec:proofNLstab}

This section is devoted to the proof of Theorem~\ref{theo:nonlinear_stabilityIntro}. 
We fix $s\in(0,1)$, $k >3$ and $\lambda \in [0 , \mu(1-s))$. We next choose $\eps_0 \in (0,\eps_*)$, where $\eps_* > 0$ is the small scale time provided by Proposition~\ref{prop:S_LL_eps}, and recall that the spaces $\XX$ and $\YY$ are defined in \eqref{eq:def:XX&YY}.

Consider the space
$$
\mathcal{Z} = \left\{ (g,w) \in  L^\infty_t (L^2_{k,0} \times (\dot H^s \cap \dot H^1)) \cap L^2_t (H^1_k \times (\dot H^s \cap \dot H^2)) \,\Big|\,  \| (g,w) \|_\mathcal{Z}<\infty \right\}
$$
with
$$
\| (g,w) \|_\mathcal{Z} := \| \mathbf{e}_{\lambda} (g,w) \|_{L^\infty_t (\XX)} +  \| \mathbf{e}_{\lambda} (g,w) \|_{L^2_t (\YY)}.
$$
For a fixed initial datum $(g_0,w_0) \in \XX$ with $\la \!\la g_0 \ra \!\ra=0$, define next  the mapping $\Phi : \mathcal{Z} \to \mathcal{Z}$, $(g,w) \mapsto \Phi[g,w]$ given by, for all $t \ge 0$,
$$
\Phi[g,w] (t) = S_\LL(t) (g_0,w_0) + (S_\LL * \RR[(g,w) , (g,w)] ) (t) ,
$$
where 
\bean
\RR [(g,w), (g,w)]  &=& \left( R_1[(g,w), (g,w)] + S_1 [(g,w), (g,w)] , \right.
\\
&&\qquad\qquad \left. R_2[(g,w), (g,w)] + {S}_2[(g,w), (g,w)]  \right),
\eean
with 
$$
\begin{aligned}
&R_1 [(g,w), (\bar g, \bar w)] = -  \Div (g \nabla \bar w), \qquad
S_1 [(g,w), (\bar g, \bar w)] = - \Div (g \nabla \kappa * \bar g), \\
&R_2 [(g,w), (\bar g, \bar w)] = \nabla \kappa *[g \nabla \bar w], \qquad
S_2 [(g,w), (\bar g, \bar w)] = \nabla \kappa * [g \nabla \kappa *\bar g] .
\end{aligned}
$$
We observe here that the first component of $\Phi [g,w] (t)$ belongs to $L^2_{k,0}$ since 
$$
\Pi R_1 [(g,w), (g,w)] = \Pi S_1 [(g,w), (g,w)] = 0,
$$ 
thus in the sequel we can apply the results of Proposition~\ref{prop:S_LL_eps}.

Thanks to Proposition~\ref{prop:S_LL_eps}, we have 
$$
 \|S_\LL( \cdot) (g_0,w_0) \|_\mathcal{Z} \lesssim \| (g_0,w_0) \|_{\XX},
$$
as well as 
$$
\begin{aligned}
\| S_\LL * \RR \|_\mathcal{Z} 
&\lesssim \| \mathbf{e}_\lambda R_1[(g,w), (g,w)] \|_{L^2_t H^{-1}_k} 
+ \| \mathbf{e}_\lambda {S}_1[(g,w), (g,w)] \|_{L^2_t H^{-1}_k} \\
&\quad
+\| \mathbf{e}_\lambda R_2 [(g,w), (g,w)]\|_{L^2_t (\dot H^s \cap \dot H^1)} 
+\| \mathbf{e}_\lambda {S}_2[(g,w), (g,w)] \|_{L^2_t (\dot H^s \cap \dot H^1)}  ,
\end{aligned}
$$
and we now estimate each term separately.

\smallskip
For the term associated to $R_1$, we first have 
$$
\begin{aligned}
\| \Div (g \nabla w) \|_{H^{-1}_k} 
&\lesssim \| g \nabla w \|_{L^2_k} \\
&\lesssim \| g \|_{L^4_k} \| \nabla w \|_{L^4} \\
&\lesssim \| g \|_{L^2_k}^{1/2} \| g \|_{H^1_k}^{1/2} \| \nabla w \|_{L^2}^{1/2} \| \nabla^2 w \|_{L^2}^{1/2} 
\\
&\lesssim \| (g,w) \|_\XX \| (g,w) \|_\YY
,
\end{aligned}
$$
where we have used H\"older's inequality in the second line, and twice the Ladyzhenskaya's inequality~\eqref{eq:Lad_ineq} in the last one. We hence obtain
\begin{equation}\label{eq:R1}
\begin{aligned}
\| \mathbf{e}_\lambda R_1[(g,w), (g,w)] \|_{L^2_t H^{-1}_k}
&\lesssim \| \mathbf{e}_\lambda g \|_{L^\infty_t L^2_k}^{1/2} \| \nabla w \|_{L^\infty_t L^2}^{1/2}  \| \mathbf{e}_\lambda g \|_{L^2_t H^1_k}^{1/2}  \|  \nabla^2 w \|_{L^2_t L^2}^{1/2} \\
&\lesssim \| (g,w) \|_{\mathcal{Z}}^2.
\end{aligned}
\end{equation}
For the term associated to $S_1$, arguing similarly as above, we get
$$
\begin{aligned}
\| \Div (g \nabla \kappa*g) \|_{H^{-1}_k} 
&\lesssim \| g \nabla \kappa *g \|_{L^2_k} \\
&\lesssim \| g \|_{L^4_k} \| \nabla \kappa *g  \|_{L^4} \\
&\lesssim \| g \|_{L^2_k}^{1/2} \| g \|_{H^1_k}^{1/2} \| \nabla \kappa *g  \|_{L^2}^{1/2} \| \nabla \kappa *g  \|_{\dot H^1}^{1/2} \\
&\lesssim \| g \|_{L^2_k} \| g \|_{H^1_k}, 
\end{aligned}
$$
where we have used  H\"older's inequality in the second line,  twice the Ladyzhenskaya's inequality~\eqref{eq:Lad_ineq} 
in the third line and finally 
Lemma~\ref{lem:borneKgHs} 
in  the last line. 
We hence obtain
\begin{equation}\label{eq:barR1}
\begin{aligned}
\| \mathbf{e}_\lambda  S_1[(g,w), (g,w)] \|_{L^2_t H^{-1}_k}
&\lesssim \| \mathbf{e}_\lambda g \|_{L^\infty_t L^2_k}   \| g \|_{L^2_t H^1_k}\\
&\lesssim \| (g,w) \|_{\mathcal{Z}}^2.
\end{aligned}
\end{equation}
For the term associated to $ R_2$, thanks to Lemma~\ref{lem:borneKgHs} we have 
$$
\begin{aligned}
\| \nabla \kappa *(g \nabla w) \|_{\dot H^s \cap \dot H^1} 
&\lesssim \| g \nabla w \|_{L^1 \cap L^2} \\
&\lesssim \| g \nabla w \|_{L^2_k},
\end{aligned}
$$
where we have used the embedding $L^2_k(\R^2) \subset L^1(\R^2)$ in last inequality.
We can thus argue as above for obtaining \eqref{eq:R1}, and we  deduce
\beqn\label{eq:R2}
\| \mathbf{e}_\lambda  R_2 [(g,w), (g,w)]\|_{L^2_t (\dot H^s \cap \dot H^1)}
\lesssim \| (g,w) \|_{\mathcal{Z}}^2.
\eeqn
Finally, for the term associated to $S_2$,  thanks to Lemma~\ref{lem:borneKgHs} we have similarly 
$$
\begin{aligned}
\| \nabla \kappa *(g \nabla \kappa *g) \|_{\dot H^s \cap \dot H^1} 
&\lesssim \| g \nabla \kappa *g \|_{L^1 \cap L^2} \\
&\lesssim \| g \nabla \kappa *g \|_{L^2_k} ,
\end{aligned}
$$
and therefore, arguing as for obtaining \eqref{eq:barR1} yields
\beqn\label{eq:S2}
\| \mathbf{e}_\lambda S_2 [(g,w), (g,w)]\|_{L^2_t (\dot H^s \cap \dot H^1)}
 \lesssim \| (g,w) \|_{\mathcal{Z}}^2.
\eeqn
Putting together \eqref{eq:R1}--\eqref{eq:S2}, we have hence obtained a first estimate
\begin{equation}\label{eq:pointfixe1}
\| \Phi[g,w] \|_{\mathcal{Z}} \le C_0 \| (g_0,w_0) \|_{\XX} + C_1 \| (g,w) \|_{\mathcal{Z}}^2,
\end{equation}
for some constants $C_0,C_1>0$.

\medskip
Now, for $(g,w), (\bar g , \bar w) \in \mathcal{Z}$, we remark that
$$
\begin{aligned}
\Phi[g,w] - \Phi[\bar g , \bar w]
&= S_\LL * \left( R_1^* + S_1^* , R_2^* + S_2^* \right)
\end{aligned}
$$
with
$$
\begin{aligned}
R_1^* &= R_1[ (g,w) , (g,w)]-R_1[ (\bar g,\bar w) , (\bar g,\bar w)]\\ 
&= R_1[ (g,w) , (g,w)-(\bar g,\bar w)]+R_1[ (g,w)-(\bar g,\bar w) , (\bar g,\bar w)]  \\
S_1^* &= S_1[ (g,w) , (g,w)]-S_1[ (\bar g,\bar w) , (\bar g,\bar w)]  \\ 
&= S_1[ (g,w) , (g,w)-(\bar g,\bar w)]+S_1[ (g,w)-(\bar g,\bar w) , (\bar g,\bar w)] \\
R_2^* &= R_2[ (g,w) , (g,w)]-R_2[ (\bar g,\bar w) , (\bar g,\bar w)]  \\
&= R_2[ (g,w) , (g,w)-(\bar g,\bar w)]+R_2[ (g,w)-(\bar g,\bar w) , (\bar g,\bar w)] \\
S_2^* &= S_2[ (g,w) , (g,w)]-S_2[ (\bar g,\bar w) , (\bar g,\bar w)] \\ 
&= S_2[ (g,w) , (g,w)-(\bar g,\bar w)]+S_2[ (g,w)-(\bar g,\bar w) , (\bar g,\bar w)] .
\end{aligned}
$$
Arguing exactly as above, we may establish a second estimate
\begin{equation}\label{eq:pointfixe2}
\| \Phi(g,w) - \Phi(\bar g , \bar w) \|_{\mathcal{Z}} \le C_2 \left( \| (g,w) \|_\mathcal{Z} + \| (\bar g , \bar w) \|_\mathcal{Z}\right) \| (g,w) - (\bar g , \bar w) \|_\mathcal{Z},
\end{equation}
for some constant $C_2>0$.

As a consequence of the estimates~\eqref{eq:pointfixe1} and~\eqref{eq:pointfixe2}, we can find $\eta_0,\eta_1> 0$ small enough such that $C_0 \eta_0 + C_1\eta_1^2 \le \eta_1$ and $2C_2\eta_1 < 1$ in such a way that $\Phi$ is a contraction on $B_\ZZ(0,\eta_1)$  for any $(g_0,w_0) \in B_\XX(0,\eta_0)$. 
By a  standard Banach fixed-point argument, one can construct a unique global mild solution $(g,w) \in \ZZ$ to \eqref{eq:KS_perturbation_gw_bis}  for any $(g_0, w_0) \in \XX$ such that $\| (g_0,w_0) \|_{\XX} \le \eta_0$.
More specifically, choosing $\eta_1 := 2C_0 \eta_0$ and $4 C_0 \max(C_1,C_2) \eta_0 < 1$, the above solution  in particular verifies the energy estimate 
\begin{equation}\label{eq:theo:energy}
\| (g,w) \|_{L^\infty_t (\XX)} + \| (g,w) \|_{L^2_t (\YY)} \le 2C_0 \| (g_0, w_0) \|_{\XX}, 
\end{equation}
which is nothing but \eqref{eq:theo:energyIntro}, as well as the decay estimate
\begin{equation}\label{eq:theo:decay}
\| \mathbf{e}_{\lambda} (g,w) \|_{L^\infty_t (\XX)} + \| \mathbf{e}_{\lambda}(g,w) \|_{L^2_t (\YY)} \le 2C_0 \| (g_0, w_0) \|_\XX,
\end{equation}
 which is nothing but \eqref{eq:theo:decayIntro}.


\bigskip


\end{document}